\documentclass{amsart}

\usepackage{hyperref}
\usepackage{amsthm,mathrsfs,amssymb,amsmath} 
\usepackage{pstricks,pst-node,pst-tree}

\begin{document}

\title[Chain enumeration of $k$-divisible noncrossing partitions]
{Chain enumeration of $k$-divisible noncrossing partitions of classical types}
\author{Jang Soo Kim}
\thanks{The author was supported by the grant ANR08-JCJC-0011.}
\email{kimjs@math.umn.edu}

\begin{abstract}
  We give combinatorial proofs of the formulas for the number of multichains in
  the $k$-divisible noncrossing partitions of classical types with certain
  conditions on the rank and the block size due to Krattenthaler and
  M{\"u}ller. We also prove Armstrong's conjecture on the zeta polynomial of the
  poset of $k$-divisible noncrossing partitions of type $A$ invariant under a
  $180^\circ$ rotation in the cyclic representation.
\end{abstract}

\newtheorem{thm}{Theorem}[section]
\newtheorem{lem}[thm]{Lemma}
\newtheorem{prop}[thm]{Proposition}
\newtheorem{cor}[thm]{Corollary}
\theoremstyle{definition}
\newtheorem{defn}{Definition}[section]
\newtheorem{conj}{Conjecture}[section]
\newtheorem{question}{Question}[section]
\newtheorem{example}{Example}[section]
\theoremstyle{remark}
\newtheorem{remark}{Remark}[section]

\newcommand\notssim[1]{\stackrel{#1}{\not\sim}}
\newcommand\ssim[1]{\stackrel{#1}{\sim}}
\newcommand\taud{\tau_D}
\newcommand\dap{\operatorname{dap}}
\newcommand\cyc{\operatorname{cyc}}
\newcommand\R{\mathbb{R}}
\newcommand\N{\mathbb{Z}^+}
\newcommand\ovpi{\overline{\pi}}
\newcommand\ovsigma{\overline{\sigma}}
\newcommand{\typek}{\operatorname{type}^{(k)}}
\newcommand{\type}{\operatorname{type}}
\newcommand{\typeb}{\overline{\operatorname{type}}}
\newcommand{\ncbb}{\mathfrak{B}}
\newcommand{\ncdd}{\mathfrak{D}}
\newcommand{\nckbb}{\mathfrak{B}^{(k)}}
\newcommand{\nckdd}{\mathfrak{D}^{(k)}}
\newcommand{\nck}{\operatorname{NC}^{(k)}}
\newcommand{\nckb}{\operatorname{NC}^{(k)}_B}
\newcommand{\nckd}{\operatorname{NC}^{(k)}_D}
\newcommand{\pdk}{P^{(k)}_D}
\newcommand{\ovpk}{\overline{P}_D^{(k)}}
\newcommand{\ovp}{\overline{P}}
\newcommand{\tnckd}{\overline{\operatorname{NC}}^{(k)}_D}
\newcommand{\nctk}{\operatorname{NC}^{(2k)}}
\newcommand{\tnc}[2]{\widetilde{\operatorname{NC}}^{(#1)}(#2)}
\newcommand{\nca}{\operatorname{NC}}
\newcommand{\NC}{\operatorname{NC}}
\newcommand{\ncb}{\operatorname{NC}_B}
\newcommand{\ncd}{\operatorname{NC}_D}
\newcommand{\ceil}[1]{\left\lceil #1 \right\rceil}
\newcommand{\floor}[1]{\left\lfloor #1 \right\rfloor}
\newcommand{\rank}{\operatorname{rank}}
\newcommand{\nz}{\operatorname{nz}}
\newcommand{\bk}{\operatorname{bk}}
\newcommand{\ind}{\operatorname{ind}}
\newcommand{\mcc}{\mathbf{c}}
\newcommand{\apa}{\overline{\alpha}}
\newcommand{\apb}{\overline{\beta}}
\newcommand{\nest}{\operatorname{nest}}
\psset{unit=0.5cm, dotsize=5pt, linewidth=0.1pt}

\def\vput#1{\pnode(#1,1){#1} \pscircle*(#1,1){.1} \rput(#1,.5){$#1$}}
\def\vvput#1#2{\pnode(#1,1){#1} \pscircle*(#1,1){.1} \rput(#1,.5){$#2$}}
\def\edge#1#2{\ncarc[arcangle=50]{#1}{#2}}
\def\activevertex#1{\pscircle[linecolor=red](#1,1){.2}}

\def\cput#1#2#3#4{\pnode(#1){#4} \pscircle*(#1){2pt} \uput[#2](#1){$#3$}}
\def\cedge#1#2{\ncline{#1}{#2}}

\newcount\ax \newcount\ay
\newcount\bx \newcount\by
\newcount\cx \newcount\cy
\newcount\dx \newcount\dy
\def\cell(#1,#2)[#3]{
\ax=#2 \ay=#1
\bx=\ax \by=\ay
\cx=\ax \cy=\ay
\dx=\ax \dy=\ay
\advance\bx by-1
\advance\dy by-1
\advance\cx by-1
\advance\cy by-1
\pspolygon(\ax,-\ay)(\bx,-\by)(\cx,-\cy)(\dx,-\dy)(\ax,-\ay)
\rput(\number\cx.5,\number-\cy.5){$#3$}
}

\def\colnum[#1,#2]{\rput[b](#1.5,.2){$#2$}}
\def\rownum[#1,#2]{\rput[r](-.2,-#1.5){$#2$}}


\maketitle

\section{Introduction}
\label{sec:introduction}

For a finite set $X$, a {\em partition} of $X$ is a collection of mutually
disjoint nonempty subsets, called \emph{blocks}, of $X$ whose union is $X$. Let
$\Pi(n)$ denote the poset of partitions of $[n]=\{1,2,\ldots,n\}$ ordered by
refinement, i.e.~$\pi\leq\sigma$ if each block of $\sigma$ is a union of some
blocks of $\pi$. There is a natural way to identify $\pi\in\Pi(n)$ with an
intersection of reflecting hyperplanes of the Coxeter group $A_{n-1}$. With this
observation Reiner \cite{Reiner1997} defined partitions of type $B_n$ as
follows.  A \emph{partition of type $B_n$} is a partition $\pi$ of $[\pm
n]=\{1,2,\ldots,n,-1,-2,\ldots,-n\}$ such that if $B$ is a block of $\pi$ then
$-B=\{-x:x\in B\}$ is also a block of $\pi$, and there is at most one block,
called \emph{zero block}, which satisfies $B=-B$. Let $\Pi_B(n)$ denote the
poset of partitions of type $B_n$ ordered by refinement.

A \emph{noncrossing partition of type $A_{n-1}$}, or simply a noncrossing partition,  is a partition $\pi\in\Pi(n)$ with the following property: if integers $a,b,c$ and $d$ with $a<b<c<d$ satisfy $a,c\in B$ and $b,d\in B'$ for some blocks $B$ and $B'$ of $\pi$, then $B=B'$. 

Let $k$ be a positive integer. 
A noncrossing partition is called \emph{$k$-divisible} if the size of each block is divisible by $k$.
Let $\nca(n)$ (resp.~$\nck(n)$) denote the subposet of $\Pi(n)$ (resp.~$\Pi(kn)$) consisting of the noncrossing partitions (resp.~$k$-divisible noncrossing partitions).

Bessis \cite{Bessis2003}, Brady and Watt \cite{Brady2008} defined the
generalized noncrossing partition poset $\NC(W)$ for each finite Coxeter group
$W$. This definition has the property that $\NC(A_{n-1})$ is isomorphic to
$\nca(n)$. Armstrong \cite{Armstrong} defined the poset $\NC^{(k)}(W)$ of
generalized $k$-divisible noncrossing partitions for each finite Coxeter group
$W$. In this definition we have $\NC^{(1)}(W)\cong \NC(W)$ and
$\NC^{(k)}(A_{n-1})\cong\nck(n)$.

For each classical Coxeter group $W$, we have a combinatorial realization of
$\NC^{(k)}(W)$.  Reiner \cite{Reiner1997} defined the poset
$\ncb(n)\subset\Pi_B(n)$ of noncrossing partitions of type $B_n$, which turned
out to be isomorphic to $\NC(B_n)$. This poset is naturally generalized to the
poset $\nckb(n) \subset\Pi_B(kn)$ of $k$-divisible noncrossing partitions of
type $B_n$. Armstrong \cite{Armstrong} showed that $\nckb(n)\cong
\NC^{(k)}(B_n)$. Athanasiadis and Reiner \cite{Athanasiadis2005} defined the
poset $\ncd(n) \subset\Pi_B(n)$ of noncrossing partitions of type $D_n$ and
showed that $\ncd(n)\cong \NC(D_n)$. Krattenthaler \cite{KrattAnnulus} defined
the poset $\nckd(n) \subset\Pi_B(kn)$ of the $k$-divisible noncrossing
partitions of type $D_n$ using a representation on an annulus and showed that
$\nckd(n)\cong \NC^{(k)}(D_n)$; see also \cite{KrattMuller}. Similarly to
$\nck(n)$, the size of each block of $\pi$ in $\nckb(n)$ or $\nckd(n)$ is
divisible by $k$.

In this paper we are mainly interested in the number of multichains in
$\nck(n)$, $\nckb(n)$ and $\nckd(n)$ with some conditions on the rank and the
block size. By a multichain we mean a chain with possible repetitions.

\begin{defn}
  For a multichain $\pi_1\leq\pi_2\leq\cdots\leq\pi_\ell$ in a graded poset $P$
  with maximum element $\hat1$, the \emph{rank jump vector} of this multichain
  is the vector $(s_1,s_2,\ldots,s_{\ell+1})$, where $s_1=\rank(\pi_1)$,
  $s_{\ell+1}=\rank(\hat1)-\rank(\pi_{\ell})$ and
  $s_i=\rank(\pi_i)-\rank(\pi_{i-1})$ for $2\leq i\leq \ell$.
\end{defn}

We note that all the posets considered in this paper are graded with maximum
element, however, they do not necessarily have a minimum element. Also note that
the sum of the integers in a rank jump vector is always equal to the rank of the
maximum element, which is $n-1$ for $\nca(n)$ and $\nck(n)$, $n$ for $\ncb(n)$,
$\ncd(n)$, $\nckb(n)$, and $\nckd(n)$. 

Edelman \cite[Theorem 4.2]{Edelman1980} showed that, if
$s_1+\cdots+s_{\ell+1}=n-1$, the number of multichains in $\nck(n)$ with rank
jump vector $(s_1, s_2,\ldots, s_{\ell+1})$ is equal to
\begin{equation}
  \label{eq:edelman}
\frac{1}{n} \binom{n}{s_1}\binom{kn}{s_2}\cdots\binom{kn}{s_{\ell+1}}.  
\end{equation}
Modifying Edelman's idea of the proof of \eqref{eq:edelman}, Reiner
\cite{Reiner1997} found an analogous formula for the number of multichains in
$\ncb(n)$ with given rank jump vector. Later, Armstrong found the following
formula \cite[Theorem 4.5.7]{Armstrong} by generalizing Reiner's idea: if
$s_1+\cdots+s_{\ell+1}=n$, the number of multichains in $\nckb(n)$ with rank
jump vector $(s_1, s_2,\ldots, s_{\ell+1})$ is equal to
\begin{equation}
  \label{eq:2}
  \binom{n}{s_1}\binom{kn}{s_2}\cdots\binom{kn}{s_{\ell+1}}.  
\end{equation}
Athanasiadis and Reiner \cite[Theorem 1.2]{Athanasiadis2005} proved that, if
$s_1+\cdots+s_{\ell+1}=n$, the number of multichains in $\ncd(n)$ with rank jump
vector $(s_1, s_2,\ldots, s_{\ell+1})$ is equal to
\begin{equation}
  \label{eq:7}
  2\binom{n-1}{s_1}\binom{n-1}{s_2}\cdots\binom{n-1}{s_{\ell+1}}
+\sum_{i=1}^{\ell+1} \binom{n-1}{s_1}\cdots \binom{n-2}{s_i-2}\cdots\binom{n-1}{s_{\ell+1}}.
\end{equation}
To prove \eqref{eq:7}, they \cite[Lemma 4.4]{Athanasiadis2005} showed the following using incidence algebras and the Lagrange inversion formula:  the number of multichains $\pi_1\leq\pi_2\leq\cdots\leq\pi_\ell$ in $\ncb(n)$ with rank jump vector $(s_1,s_2,\ldots,s_{\ell+1})$ such that $d$ is the smallest integer for which $\pi_d$ has a zero block is equal to
\begin{equation}
  \label{eq:14}
\frac{s_d}{n} \binom{n}{s_1}\binom{n}{s_2} \cdots \binom{n}{s_{\ell+1}}.  
\end{equation}
Since \eqref{eq:14} is simple and elegant, it deserves a combinatorial proof. In this paper we prove a generalization of \eqref{eq:14} combinatorially;  see Lemma~\ref{thm:atha}.

The number of noncrossing partitions with given block sizes has been studied as well. In the literature, for instance \cite{ Armstrong, Athanasiadis1998, Athanasiadis2005}, $\type(\pi)$ for $\pi\in\Pi(n)$ is defined to be the integer partition $\lambda=(\lambda_1,\lambda_2,\ldots,\lambda_\ell)$ where the number of parts of size $i$ is equal to the number of blocks of size $i$ of $\pi$. However, to state the results in a uniform way, we will use the following different definition of $\type(\pi)$.

\begin{defn}\label{def:type}
  The \emph{type} of a partition $\pi\in\Pi(n)$, denoted by $\type(\pi)$, is the
  sequence $(b;b_1,b_2,\ldots,b_n)$ where $b_i$ is the number of blocks of $\pi$
  of size $i$ and $b=b_1+b_2+\cdots+b_n$. The \emph{type} of $\pi\in\Pi_B(n)$,
  denoted by $\type(\pi)$, is the sequence $(b;b_1,b_2,\ldots,b_n)$ where $b_i$
  is the number of unordered pairs $(B,-B)$ of nonzero blocks of $\pi$ of size
  $i$ and $b=b_1+b_2+\cdots+b_n$.  For a partition $\pi$ in either $\Pi(kn)$ or
  $\Pi_B(kn)$, if $\type(\pi)=(b;b_1,b_2,\ldots,b_{kn})$ and the size of each
  block of $\pi$ is divisible by $k$, then we define the \emph{$k$-type}
  $\typek(\pi)$ of $\pi$ to be $(b;b_{k},b_{2k},\ldots,b_{kn})$.
\end{defn}

From now on, whenever we write $(b;b_1,b_2,\ldots,b_n)$, it is automatically
assumed that $b=b_1+b_2+\cdots+b_n$.

Note that if $\type(\pi)=(b;b_1,b_2,\ldots,b_n)$ for $\pi$ in $\nca(n)$
(resp.~$\ncb(n)$), then we have $\sum_{i=1}^n i \cdot b_i =n$
(resp.~$\sum_{i=1}^n i\cdot b_i \leq n$). Furthermore, if
$\type(\pi)=(b;b_1,b_2,\ldots,b_n)$ for $\pi$ in $\ncd(n)$, then $\sum_{i=1}^n
i\cdot b_i =n$ or $\sum_{i=1}^n i\cdot b_i \leq n-2$ because $\pi$ cannot have a
zero block of size $1$.

Kreweras \cite[Theorem 4]{Kreweras1972} proved that, if $\sum_{i=1}^ni\cdot b_i
=n$, the number of $\pi\in\nca(n)$ with $\type(\pi)=(b;b_1,b_2,\ldots,b_n)$ is
equal to
\begin{equation}
  \label{eq:8}
\frac{1}{b}\binom{b}{b_1,b_2,\ldots,b_n}\binom{n}{b-1}.
\end{equation}
Athanasiadis \cite[Theorem 2.3]{Athanasiadis1998} proved that, if
$\sum_{i=1}^ni\cdot b_i \leq n$, the number of $\pi\in\ncb(n)$ with
$\type(\pi)=(b;b_1,b_2,\ldots,b_n)$ is equal to
\begin{equation}
  \label{eq:9}
\binom{b}{b_1,b_2,\ldots,b_n}\binom{n}{b-1}.
\end{equation}
Athanasiadis and Reiner \cite[Theorem 1.3]{Athanasiadis2005} proved that the
number of $\pi\in\ncd(n)$ with $\type(\pi)=(b;b_1,b_2,\ldots,b_n)$ is equal to
\begin{equation}
  \label{eq:10}
\binom{b}{b_1,b_2,\ldots,b_n}\binom{n-1}{b-1},
\end{equation}
if $b_1+2b_2+\cdots+n b_n\leq n-2$, and
\begin{equation}
  \label{eq:11}
2\binom{b}{b_1,b_2,\ldots,b_n}\binom{n-1}{b}
+\binom{b-1}{b_1-1,b_2,\ldots,b_n}\binom{n-1}{b-1},
\end{equation}
if $b_1+2b_2+\cdots+n b_n=n$.

Armstrong \cite[Theorem 4.4.4 and Theorem 4.5.11]{Armstrong} generalized
\eqref{eq:8} and \eqref{eq:9} as follows: if $\sum_{i=1}^ni\cdot b_i =n$, the
number of multichains $\pi_1\leq\pi_2\leq\cdots\leq\pi_{\ell}$ in $\nck(n)$ with
$\typek(\pi_1)=(b;b_1,b_2,\ldots,b_n)$ is equal to
\begin{equation}
  \label{eq:12}
\frac{1}{b}\binom{b}{b_1,b_2,\ldots,b_n}\binom{\ell kn}{b-1},
\end{equation}
and, if $\sum_{i=1}^ni\cdot b_i \leq n$, the number of multichains
$\pi_1\leq\pi_2\leq\cdots\leq\pi_{\ell}$ in $\nckb(n)$ with
$\typek(\pi_1)=(b;b_1,b_2,\ldots,b_n)$ is equal to
\begin{equation}
  \label{eq:13}
\binom{b}{b_1,b_2,\ldots,b_n}\binom{\ell kn}{b}.
\end{equation}
 
Using decomposition numbers, Krattenthaler and M{\"u}ller \cite{KrattMuller}
proved certain very general formulas and as corollaries, they obtained the
following three theorems which generalize all the above known results except
\eqref{eq:14}.  In fact, the results in \cite{KrattMuller} are far more general
than the following three theorems and our generalization of \eqref{eq:14} can
also be obtained from one of those.

\begin{thm}\cite[Corollary 12]{KrattMuller}\label{thm:kratt1}
Let $b,b_1,b_2,\ldots,b_n$ and $s_1,s_2,\ldots,s_{\ell+1}$ be nonnegative integers satisfying $\sum_{i=1}^n b_i=b$, $\sum_{i=1}^n i\cdot b_i=n$, $\sum_{i=1}^{\ell+1} s_i = n-1$ and $s_1=n-b$. Then the number of multichains $\pi_1\leq \pi_2\leq \cdots \leq \pi_{\ell}$ in $\nck(n)$ with rank jump vector $(s_1,s_2,\ldots,s_{\ell+1})$ and $\typek(\pi_1)=(b;b_1,b_2,\ldots,b_n)$ is equal to
$$\frac{1}{b} \binom{b}{b_1,b_2,\ldots,b_n} \binom{kn}{s_2} \cdots \binom{kn}{s_{\ell+1}}.$$
\end{thm}

\begin{thm}\cite[Corollary 14]{KrattMuller}\label{thm:kratt2}
Let $b,b_1,b_2,\ldots,b_n$ and $s_1,s_2,\ldots,s_{\ell+1}$ be nonnegative integers satisfying $\sum_{i=1}^n b_i=b$, $\sum_{i=1}^n i\cdot b_i\leq n$, $\sum_{i=1}^{\ell+1} s_i = n$ and $s_1=n-b$. 
Then the number of multichains $\pi_1\leq \pi_2\leq \cdots \leq \pi_{\ell}$ in $\nckb(n)$ with rank jump vector $(s_1,s_2,\ldots,s_{\ell+1})$ and $\typek(\pi_1)=(b;b_1,b_2,\ldots,b_n)$ is equal to
$$\binom{b}{b_1,b_2,\ldots,b_n} \binom{kn}{s_2} \cdots \binom{kn}{s_{\ell+1}}.$$
\end{thm}

\begin{thm}\cite[Corollary 16]{KrattMuller}\label{thm:kratt3}
Let $b,b_1,b_2,\ldots,b_n$ and $s_1,s_2,\ldots,s_{\ell+1}$ be nonnegative integers satisfying $\sum_{i=1}^n b_i=b$, $\sum_{i=1}^n i\cdot b_i\leq n$, $\sum_{i=1}^n i \cdot b_i \ne n-1$, $\sum_{i=1}^{\ell+1} s_i = n$ and $s_1=n-b$. 
Then the number of multichains $\pi_1\leq \pi_2\leq \cdots \leq \pi_{\ell}$ in $\nckd(n)$ with rank jump vector $(s_1,s_2,\ldots,s_{\ell+1})$ and $\typek(\pi_1)=(b;b_1,b_2,\ldots,b_n)$ is equal to
$$\binom{b}{b_1,b_2,\ldots,b_n} \binom{k(n-1)}{s_2} \cdots \binom{k(n-1)}{s_{\ell+1}},$$
if $b_1+2b_2+\cdots+n b_n\leq n-2$, and
\begin{multline*}
 2\binom{b}{b_1,b_2,\ldots,b_n} \binom{k(n-1)}{s_2} \cdots \binom{k(n-1)}{s_{\ell+1}}\\
 + \sum_{i=2}^{\ell+1} \frac{s_i-1}{b-1} \binom{b-1}{b_1-1,b_2,\ldots,b_n} 
\binom{k(n-1)}{s_2} \cdots \binom{k(n-1)}{s_i-1} \cdots \binom{k(n-1)}{s_{\ell+1}},
\end{multline*}
if $b_1+2b_2+\cdots+n b_n= n$.
\end{thm}

Krattenthaler and M{\"u}ller's proofs of Theorems~\ref{thm:kratt1},
\ref{thm:kratt2} and \ref{thm:kratt3} are not completely combinatorial because
they used the Lagrange-Good inversion formula, see \cite[Theorem 1]{KrattMuller}
or \cite{Good1960}. Especially, in the introduction in \cite{KrattMuller}, they
wrote that Theorems~\ref{thm:kratt1} and \ref{thm:kratt2} seem amenable to
combinatorial proofs, however, to find a combinatorial proof of
Theorem~\ref{thm:kratt3} seems rather hopeless.  In this paper, we will give
combinatorial proofs of Theorems~\ref{thm:kratt1}, \ref{thm:kratt2} and
\ref{thm:kratt3}.

This paper is organized as follows. In Section~\ref{sec:interpr-noncr-part} we
recall the definition of $\ncb(n)$ and $\ncd(n)$. In
Section~\ref{sec:interpr-noncr-part-2} we recall the bijection $\psi$ in
\cite{KimNoncrossing2} between $\ncb(n)$ and the set $\ncbb(n)$ of pairs
$(\sigma,x)$, where $\sigma\in\nca(n)$ and $x$ is either $\emptyset$, an edge of
$\sigma$, or a block of $\sigma$. Here an \emph{edge} of $\sigma$ is a pair
$(i,j)$ of integers with $i<j$ such that $i$ and $j$ are in the same block of
$\sigma$ which does not contain any integer between them.  Then we find a
necessary and sufficient condition for two pairs $(\sigma_1,x_1)\in\ncbb(n)$ and
$(\sigma_2,x_2)\in\ncbb(n)$ to satisfy $\psi^{-1}(\sigma_1,x_1)\leq
\psi^{-1}(\sigma_2,x_2)$ in the poset $\ncb(n)$. This property plays a crucial
role in the proof of Theorem~\ref{thm:kratt3}. In
Section~\ref{sec:k-divis-noncr} we prove Theorem~\ref{thm:kratt1} by modifying
the argument of Edelman \cite{Edelman1980}. For $0<r<k$, we consider the
subposet $\nck(n;r)$ of $\nca(nk+r)$ consisting of the partitions $\pi$ such
that all but one block of $\pi$ have sizes divisible by $k$. Then we prove a
chain enumeration formula for $\nck(n;r)$, which is similar to
Theorem~\ref{thm:kratt1}. In Section~\ref{sec:kb} we prove a generalization of
\eqref{eq:14} and Theorem~\ref{thm:kratt2}. In Section~\ref{sec:armstr-conj} we
prove that the poset $\tnc{2k}{2n+1}$ suggested by Armstrong \cite{Armstrong} is
isomorphic to the poset $\nctk(n;k)$ defined in
Section~\ref{sec:k-divis-noncr}. Using these results, we prove Armstrong's
conjecture on the zeta polynomial of $\tnc{2k}{2n+1}$ and answer the question on
rank-, type-selection formulas \cite[Conjecture 4.5.14 and Open Problem
4.5.15]{Armstrong}. In Section~\ref{sec:kd} we prove Theorem~\ref{thm:kratt3}.
All the arguments in this paper are purely combinatorial.

\section{Noncrossing partitions of classical types}
 \label{sec:interpr-noncr-part}

For integers $n,m$, let $[n,m]=\{n,n+1,\ldots,m\}$ and $[n]=\{1,2,\ldots,n\}$.

 Recall that $\Pi(n)$ denotes the poset of partitions of $[n]$ and $\Pi_B(n)$
 denotes the poset of partitions of type $B_n$.  For simplicity, we will write a
 partition $\pi\in\Pi_B(n)$ in the following way:
$$\pi= \{ \pm \{1,-3,6\}, \{2,4,-2,-4\}, \pm\{5,8\}, \pm\{7\} \},$$ which means 
$$\pi= \{ \{1,-3,6\}, \{-1,3,-6\}, \{2,4,-2,-4\}, \{5,8\}, \{-5,-8\}, \{7\}, \{-7\} \}.$$

The \emph{circular representation} of $\pi\in\Pi(n)$ is the drawing obtained as
follows. Arrange $n$ vertices around a circle which are labeled with the
integers $1,2,\ldots,n$. For each block $B$ of $\pi$, draw the convex hull of
the vertices whose labels are the integers in $B$. For an example, see
Figure~\ref{fig:circular}. It is easy to see that the following definition
coincides with the definition of a noncrossing partition in the introduction:
$\pi$ is a \emph{noncrossing partition} if the convex hulls in the circular
representation of $\pi$ do not intersect.

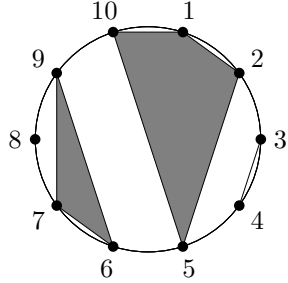
\begin{figure}
  \centering
\psset{unit=.3cm,linewidth=0.1pt}
\begin{pspicture}(-5,-5)(5,5)
  \pscircle(0,0){5}
\psline(5.000,0.0000)(4.045,-2.939)
\pscircle(0,0){5}
\pspolygon[fillstyle=solid,fillcolor=gray](-1.545,4.755)(1.545,4.755)(4.045,2.939)(1.545,-4.755)
\pscircle(0,0){5}
\pspolygon[fillstyle=solid,fillcolor=gray](-1.545,-4.755)(-4.045,-2.939)(-4.045,2.939)
\cput{1.545,4.755}{72.00}{1}{1}
\cput{4.045,2.939}{36.00}{2}{2}
\cput{5.000,0.0000}{0.0000}{3}{3}
\cput{4.045,-2.939}{-36.00}{4}{4}
\cput{1.545,-4.755}{-72.00}{5}{5}
\cput{-1.545,-4.755}{-108.0}{6}{6}
\cput{-4.045,-2.939}{-144.0}{7}{7}
\cput{-5.000,0.0000}{-180.0}{8}{8}
\cput{-4.045,2.939}{-216.0}{9}{9}
\cput{-1.545,4.755}{-252.0}{10}{10}
\end{pspicture}
 \caption{The circular representation of $\{\{1,2,5,10\}, \{3,4\}, \{6,7,9\}, \{8\}\}$.}
  \label{fig:circular}
\end{figure}

Let $\pi\in\Pi_B(n)$. The \emph{type $B$ circular representation} of $\pi$ is
the drawing obtained as follows. Arrange $2n$ vertices on a circle which are
labeled with the integers $1,2,\ldots,n,-1,-2,\ldots,-n$. For each block $B$ of
$\pi$, draw the convex hull of the vertices whose labels are the integers in
$B$. For an example, see Figure~\ref{fig:circular-B}. Note that the type $B$
circular representation of $\pi\in\Pi_B(n)$ is invariant, if we do not concern
the labels, under a $180^\circ$ rotation, and the zero block of $\pi$, if it
exists, corresponds to a convex hull containing the center. Then $\pi$ is a
\emph{noncrossing partition of type $B_n$} if the convex hulls do not intersect.

\begin{figure}
  \centering
\psset{unit=.3cm,linewidth=0.1pt}
  \begin{pspicture}(-5,-5)(5,5)
\pscircle(0,0){5}
\pspolygon[fillstyle=solid,fillcolor=gray](1.545,4.755)(4.045,-2.939)(-1.545,4.755)
\pspolygon[fillstyle=solid,fillcolor=gray](-1.545,-4.755)(-4.045,2.939)(1.545,-4.755)
\pscircle(0,0){5}
\psline(4.045,2.939)(5.000,0.0000)
\psline(-4.045,-2.939)(-5.000,0.0000)
\cput{1.545,4.755}{72.00}{1}{1}
\cput{-1.545,-4.755}{252.0}{-1}{-1}
\cput{4.045,2.939}{36.00}{2}{2}
\cput{-4.045,-2.939}{216.0}{-2}{-2}
\cput{5.000,0.0000}{0.0000}{3}{3}
\cput{-5.000,0.0000}{180.0}{-3}{-3}
\cput{4.045,-2.939}{-36.00}{4}{4}
\cput{-4.045,2.939}{144.0}{-4}{-4}
\cput{1.545,-4.755}{-72.00}{5}{5}
\cput{-1.545,4.755}{108.0}{-5}{-5}
\end{pspicture}
  \begin{pspicture}(-9,-5)(5,5)
\pscircle(0,0){5}
\pspolygon[fillstyle=solid,fillcolor=gray](-1.545,-4.755)(-4.045,2.939)(-1.545,4.755)(1.545,4.755)(4.045,-2.939)(1.545,-4.755)
\pscircle(0,0){5}
\psline(4.045,2.939)(5.000,0.0000)
\psline(-4.045,-2.939)(-5.000,0.0000)
\cput{1.545,4.755}{72.00}{1}{1}
\cput{-1.545,-4.755}{252.0}{-1}{-1}
\cput{4.045,2.939}{36.00}{2}{2}
\cput{-4.045,-2.939}{216.0}{-2}{-2}
\cput{5.000,0.0000}{0.0000}{3}{3}
\cput{-5.000,0.0000}{180.0}{-3}{-3}
\cput{4.045,-2.939}{-36.00}{4}{4}
\cput{-4.045,2.939}{144.0}{-4}{-4}
\cput{1.545,-4.755}{-72.00}{5}{5}
\cput{-1.545,4.755}{108.0}{-5}{-5}
  \end{pspicture}
  \caption{The type $B$ circular representations of $\{\pm\{1,4,-5\},\pm\{2,3\}\}$ and 
$\{\{1,4,5,-1,-4,-5\},\pm\{2,3\}\}$.}
  \label{fig:circular-B}
\end{figure}
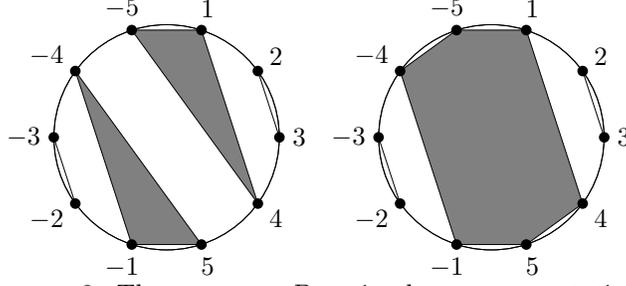

Let $\pi\in\Pi_B(n)$. The \emph{type $D$ circular representation} of $\pi$ is
the drawing obtained as follows. Arrange $2n-2$ vertices labeled with
$1,2,\ldots,n-1,-1,-2,\ldots,-(n-1)$ on a circle and put a vertex labeled with
$\pm n$ at the center.  For each block $B$ of $\pi$, draw the convex hull of the
vertices whose labels are in $B$. Then $\pi$ is a \emph{noncrossing partition of
  type $D_n$} if the convex hulls do not intersect in their interiors and if
there is a zero block $B$ then $\{n,-n\}\subsetneq B$. For an example, see
Figure~\ref{fig:circ-D}.

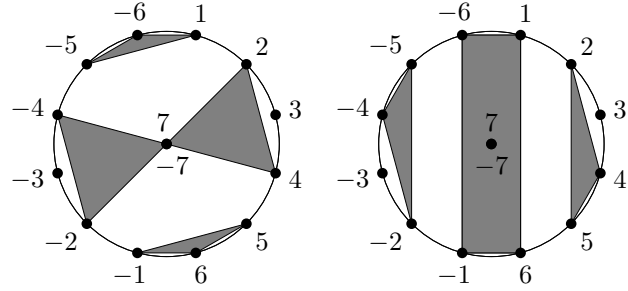
\begin{figure}
  \centering
\psset{unit=.3cm,linewidth=0.1pt}
  \begin{pspicture}(-5,-5)(5,5)
\pscircle[fillstyle=solid](0,0){0}
\pscircle(0,0){5}
\pspolygon[fillstyle=solid,fillcolor=gray](3.536,3.536)(4.830,-1.294)(0.0000,0.0000)
\pspolygon[fillstyle=solid,fillcolor=gray](-3.536,-3.536)(-4.830,1.294)(0.0000,0.0000)
\pscircle[fillstyle=solid](0,0){0}
\pscircle(0,0){5}
\pspolygon[fillstyle=solid,fillcolor=gray](3.536,-3.536)(1.294,-4.830)(-1.294,-4.830)
\pspolygon[fillstyle=solid,fillcolor=gray](-3.536,3.536)(-1.294,4.830)(1.294,4.830)
\cput{1.294,4.830}{75.00}{1}{1}
\cput{-1.294,-4.830}{255.0}{-1}{-1}
\cput{3.536,3.536}{45.00}{2}{2}
\cput{-3.536,-3.536}{225.0}{-2}{-2}
\cput{4.830,1.294}{15.00}{3}{3}
\cput{-4.830,-1.294}{195.0}{-3}{-3}
\cput{4.830,-1.294}{-15.00}{4}{4}
\cput{-4.830,1.294}{165.0}{-4}{-4}
\cput{3.536,-3.536}{-45.00}{5}{5}
\cput{-3.536,3.536}{135.0}{-5}{-5}
\cput{1.294,-4.830}{-75.00}{6}{6}
\cput{-1.294,4.830}{105.0}{-6}{-6}
\cput{0.0000,0.0000}{100}{7}{7}
\cput{0.0000,0.0000}{280}{-7}{-7}  
  \end{pspicture}
 \begin{pspicture}(-9,-5)(5,5)
\pscircle[fillstyle=solid](0,0){0}
\pscircle(0,0){5}
\pspolygon[fillstyle=solid,fillcolor=gray](1.294,4.830)(1.294,-4.830)(-1.294,-4.830)(-1.294,4.830)
\pspolygon[fillstyle=solid,fillcolor=gray](-1.294,-4.830)(-1.294,4.830)(1.294,4.830)(1.294,-4.830)
\pscircle[fillstyle=solid](0,0){0}
\pscircle(0,0){5}
\pspolygon[fillstyle=solid,fillcolor=gray](3.536,3.536)(4.830,-1.294)(3.536,-3.536)
\pspolygon[fillstyle=solid,fillcolor=gray](-3.536,-3.536)(-4.830,1.294)(-3.536,3.536)
\cput{1.294,4.830}{75.00}{1}{1}
\cput{-1.294,-4.830}{255.0}{-1}{-1}
\cput{3.536,3.536}{45.00}{2}{2}
\cput{-3.536,-3.536}{225.0}{-2}{-2}
\cput{4.830,1.294}{15.00}{3}{3}
\cput{-4.830,-1.294}{195.0}{-3}{-3}
\cput{4.830,-1.294}{-15.00}{4}{4}
\cput{-4.830,1.294}{165.0}{-4}{-4}
\cput{3.536,-3.536}{-45.00}{5}{5}
\cput{-3.536,3.536}{135.0}{-5}{-5}
\cput{1.294,-4.830}{-75.00}{6}{6}
\cput{-1.294,4.830}{105.0}{-6}{-6}
\cput{0.0000,0.0000}{90}{7}{7}
\cput{0.0000,0.0000}{270}{-7}{-7}
  \end{pspicture}
  \caption{The type $D$ circular representations of $\{ \pm\{1,-5,-6\}$,
    $\pm\{2,4,-7\}$, $\pm\{3\} \}$ and $\{ \{1,6,7,-1,-6,-7\}$, $\pm\{2,4,5\}$,
    $\pm\{3\} \}$.}
  \label{fig:circ-D}
\end{figure}

Let $\pi\in\Pi(n)$. An \emph{edge} of $\pi$ is a pair $(i,j)$ of integers with
$i<j$ such that $i,j\in B$ for a block $B$ of $\pi$ and there is no other
integer $k$ in $B$ with $i<k<j$. The \emph{standard representation} of $\pi$ is
the drawing obtained as follows. Arrange the integers $1,2,\ldots,n$ on a
horizontal line. For each edge $(i,j)$ of $\pi$, connect the integers $i$ and
$j$ with an arc above the horizontal line. For an example, see
Figure~\ref{fig:standard}. Then $\pi$ is a noncrossing partition if and only if
the arcs in the standard representation do not intersect.

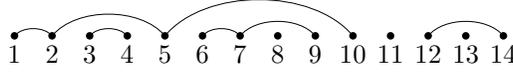
\begin{figure}
  \centering
  \begin{pspicture}(1,0.5)(14,2) \vput{1} \vput{2} \vput{3} \vput{4}  \vput{5} \vput{6} \vput{7} \vput{8} \vput{9} \vput{10} \vput{11} \vput{12} \vput{13} \vput{14} 
    \edge12 \edge25 \edge5{10} \edge34 \edge67 \edge79 \edge{12}{14}
  \end{pspicture}
  \caption{The standard representation of $\{\{1,2,5,10\}$, $\{3,4\}$,
    $\{6,7,9\}$, $\{8\}$, $\{11\}$, $\{12,14\}$, $\{13\}\}$.}
  \label{fig:standard}
\end{figure}

Let $\pi\in\Pi_B(n)$. The \emph{standard representation} of $\pi$ is the drawing
obtained as follows. Arrange the integers $1,2,\ldots,n,-1,-2,\ldots,-n$ on a
horizontal line. Then connect the integers $i$ and $j$ with an arc above the
horizontal line for each pair $(i,j)$ of integers such that $i$ and $j$ are in
the same block $B$ of $\pi$ and there is no other integer in $B$ between $i$ and
$j$ on the horizontal line. For an example, see Figure~\ref{fig:stand-B}. Then
$\pi$ is a noncrossing partition of type $B_n$ if and only if the arcs in the
standard representation do not intersect.

Let $\ncb(n)$ (resp.~$\ncd(n)$) denote the subposet of $\Pi_B(n)$ consisting of
the noncrossing partitions of type $B_n$ (resp.~type $D_n$). Note that for
$\pi\in\ncb(n)$ or $\pi\in\ncd(n)$, we have $\rank(\pi)=n-\nz(\pi)$, where
$\nz(\pi)$ denotes the number of unordered pairs $(B,-B)$ of nonzero blocks of
$\pi$.

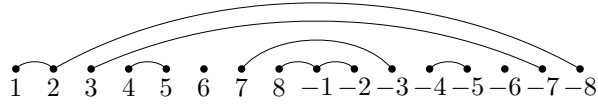
\begin{figure}
  \centering
  \begin{pspicture}(1,0.5)(16,2.5) 
\vput{1} \vput{2} \vput{3} \vput{4}  \vput{5} \vput{6} \vput7 \vput8
\vvput{9}{-1} \vvput{10}{-2} \vvput{11}{-3} \vvput{12}{-4} \vvput{13}{-5} \vvput{14}{-6} \vvput{15}{-7} \vvput{16}{-8} 
\edge12 \edge45
\edge{9}{10} \edge{12}{13}
\ncarc[arcangle=30]{2}{16}
\ncarc[arcangle=25]{3}{15}
\edge{7}{11} \edge89
 \end{pspicture}
\caption{The standard representation of $\{\pm\{1,2,-8\},\pm\{3,-7\}, \pm\{4,5\}, \pm\{6\}\}$.}
  \label{fig:stand-B}
\end{figure}

\section{Interpretation of noncrossing partitions of type $B_n$}
\label{sec:interpr-noncr-part-2}

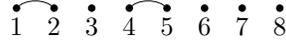
\begin{figure}
  \centering
  \begin{pspicture}(1,0.5)(8,1) 
\vput{1} \vput{2} \vput{3} \vput{4}  \vput{5} \vput{6} \vput7 \vput8
\edge12 \edge45 
\end{pspicture}
\caption{The partition $\eta$ corresponding to the partition $\pi$ in
  Figure~\ref{fig:stand-B}.}
  \label{fig:phi}
\end{figure}

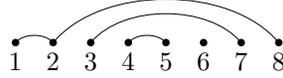
\begin{figure}
  \centering
  \begin{pspicture}(1,0.5)(8,2) 
\vput{1} \vput{2} \vput{3} \vput{4}  \vput{5} \vput{6} \vput7 \vput8
\edge12 \edge45
\edge37 \edge28
\end{pspicture}
\caption{The partition $\sigma$ corresponding to the partition $\pi$ in
  Figure~\ref{fig:stand-B}.}
  \label{fig:psi}
\end{figure}

Let $\ncbb(n)$ denote the set of pairs $(\sigma, x)$, where $\sigma\in\nca(n)$
and $x$ is either $\emptyset$, an edge of $\sigma$, or a block of $\sigma$. Note
that for each $\sigma\in\nca(n)$, we have $n+1$ choices for $x$ with
$(\sigma,x)\in\ncbb(n)$. Thus $\ncbb(n)$ is essentially the same as
$\nca(n)\times[n+1]$.

Let us recall the bijection $\psi:\ncb(n)\rightarrow\ncbb(n)$ in \cite{KimNoncrossing2}.

Given $\pi\in\ncb(n)$, let $\eta$ be the partition of $[n]$ obtained from $\pi$
by removing all the negative integers and let $X$ be the set of blocks $B$ of
$\eta$ such that $\pi$ has a block containing $B$ and at least one negative
integer. Note that $\eta\in\nca(n)$ and each block $B$ in $X$ is
\emph{nonnested} in $\eta$, i.e.~$\eta$ does not have an edge $(i,j)$ with
$i<\min(B)\leq\max(B)<j$. For example, if $\pi$ is the partition in
Figure~\ref{fig:stand-B}, then $\eta$ is the partition in Figure~\ref{fig:phi}
and $X=\{\{1,2\}, \{3\}, \{7\}, \{8\}\}$.

Now assume that $X$ has $k$ blocks $A_1,A_2,\ldots,A_k$ with
$\max(A_1)<\max(A_2)<\cdots<\max(A_k)$. Since all the blocks in $X$ are
nonnested in $\eta$, we have $\max(A_i)<\min(A_{i+1})$ for all $i\in[k-1]$.  Let
$\sigma$ be the partition obtained from $\eta$ by taking the union of $A_i$ and
$A_{k+1-i}$ for all $i=1,2,\ldots,\floor{k/2}$. Let
$$x = \left\{
  \begin{array}{ll}
    \emptyset, & \mbox{if $k=0$;} \\
    (\max(A_{k/2}), \min(A_{k/2+1})), & \mbox{if $k\ne0$ and $k$ is even;}\\
    A_{(k+1)/2}, & \mbox{if $k$ is odd.}
  \end{array} \right. $$
Note that $(\sigma,x)\in\ncbb(n)$. We define $\psi(\pi)$ to be the pair
$(\sigma,x)$.  For example, if $\pi$ is the partition in
Figure~\ref{fig:stand-B}, then $\psi(\pi) = (\sigma,x)$, where $\sigma$ is the
partition in Figure~\ref{fig:psi} and $x$ is the edge $(3,7)$.

The following theorem is proved by the author in \cite{KimNoncrossing2}. For
the sake of self-containedness, we include a short proof.

\begin{thm}\label{thm:simbij} \cite{KimNoncrossing2}
  The map $\psi:\ncb(n)\to\ncbb(n)$ is a bijection. Moreover, for
  $\pi\in\ncb(n)$ with $\type(\pi)=(b;b_1,b_2,\ldots,b_n)$ and
  $\psi(\pi)=(\sigma,x)$, we have $\type(\sigma)=\type(\pi)$ if $\pi$ does not
  have a zero block; and $\type(\sigma)=(b+1;b_1,\ldots,b_i+1,\ldots,b_n)$ if
  $\pi$ has a zero block of size $2i$.
\end{thm}
\begin{proof}
  It is sufficient to construct the inverse map $\psi^{-1}$. Let
  $(\sigma,x)\in\ncbb(n)$. Then $\psi^{-1}(\sigma,x)$ is the partition
  $\pi\in\ncb(n)$ determined as follows.
\begin{description}
\item[Case 1] $x$ is the empty set $\emptyset$. Then for each
  block $A$ of $\sigma$, $\pm A$ are blocks of $\pi$.
\item[Case 2] $x$ is an edge $(a,b)$ of $\sigma$. For each block $A$ of
  $\sigma$, let $A_1=A\cap[a]$ and $A_2=A\cap [b,n]$. If $A_1=\emptyset$ or
  $A_2=\emptyset$, then $\pm A$ are blocks of $\pi$. Otherwise, $\pm(A_1 \cup
  (-A_2))$ are blocks of $\pi$.
\item[Case 3] $x$ is a block $B$ of $\sigma$. Then $\pi$ has zero block
  $B\cup(-B)$. For each block $A\ne B$ of $\sigma$, let $A_1=A\cap[\min(B)]$ and
  $A_2=A\cap [\max(B),n]$. If $A_1=\emptyset$ or $A_2=\emptyset$, then $\pm A$
  are blocks of $\pi$. Otherwise, $\pm(A_1 \cup (-A_2))$ are blocks of $\pi$.
\end{description}
It is easy to check that this is the inverse map of $\psi$. 
The `moreover' statement is clear from the construction of $\psi^{-1}$. 
\end{proof}

Now we will find a necessary and sufficient condition for $(\sigma_1,x_1)$,
$(\sigma_2,x_2)\in\ncbb(n)$ to satisfy $\psi^{-1}(\sigma_1,x_1) \leq
\psi^{-1}(\sigma_2,x_2)$ in $\ncb(n)$.

For a partition $\pi$ in $\Pi(n)$ or $\Pi_B(n)$, we write $i\ssim{\pi}j$ if $i$
and $j$ are in the same block of $\pi$ and $i\notssim{\pi}j$ otherwise.  Note
that if $\psi(\pi)=(\sigma,x)$, then we have $i\ssim{\sigma} j$ if and only if
$i\ssim{\pi} j$ or $i\ssim\pi -j$.  The \emph{length} of an edge $(i,j)$ is
defined to be $j-i$. It is straightforward to check the following lemmas.

\begin{lem}\label{lem:t1}
  Let $\psi(\pi)=(\sigma,x)$. Then we have the following properties.
  \begin{enumerate}
  \item $x=\emptyset$ if and only if there is no integer $i$ with $i\ssim\pi -i$
    and there are no positive integers $i,j$ with $i\ssim\pi -j$.
  \item $x$ is an edge if and only if there is no integer $i$ with $i\ssim\pi
    -i$ and there are two positive integers $i,j$ with $i\ssim\pi -j$. In this
    case, there is a unique pair $(a,b)$ of integers $a<b$ such that $b-a$ is
    minimal subject to the condition $a\ssim\pi -b$. Then $(a,b)$ is in fact an
    edge of $\sigma$ and we have $x=(a,b)$.
  \item $x$ is a block if and only if there is an integer $i$ with $i\ssim\pi -i$. In this case, $x$ is the block consisting of the positive integers $i$ with $i\ssim\pi -i$.
 \end{enumerate}
\end{lem}

\begin{lem}\label{lem:t2}
  Let $\psi(\pi)=(\sigma,x)$. For any integers $i,j$ with $i<j$ and
  $i\ssim{\sigma} j$, we have the following properties.
  \begin{enumerate}
  \item If $x=\emptyset$, then $i\ssim{\pi}j$ and $i\notssim{\pi}-j$.
  \item If $x$ is an edge $(a,b)$ of $\sigma$, then $i\notssim\pi j$ and $i\ssim\pi -j$ if $i\leq a<b\leq j$; and $i\ssim\pi j$ and $i\notssim\pi -j$ otherwise.
  \item If $x$ is a block $B$ of $\sigma$, then $i\ssim\pi j$ and $i\ssim\pi -j$
    if $\{i,j\}\subset B$; $i\notssim\pi j$ and $i\ssim\pi -j$ if
    $\{i,j\}\not\subset B$ and $i< \min B\leq \max B< j$; and $i\ssim\pi j$ and
    $i\notssim\pi -j$ otherwise.
\end{enumerate}
\end{lem}

\begin{prop}\label{thm:12}
  Let $\psi(\pi_1)=(\sigma_1,x_1)$ and $\psi(\pi_2)=(\sigma_2,x_2)$. Then
  $\pi_1\leq\pi_2$ if and only if $\sigma_1\leq\sigma_2$ and one of the
  following conditions holds:
 \begin{enumerate}
  \item\label{item:1} $x_1=x_2=\emptyset$.
\item\label{item:3} $x_2$ is an edge $(a,b)$ of $\sigma_2$ and $x_1$ is the unique minimal length edge $(i,j)$ of $\sigma_1$ with $i\leq a<b\leq j$ if such an edge exists; and $x_1=\emptyset$ otherwise. 
\item\label{item:2} $x_2$ is a block of $\sigma_2$ and $x_1$ is a block of
  $\sigma_1$ with $x_1\subset x_2$.
  \item\label{item:5} $x_2$ is a block of $\sigma_2$ and $x_1$ is an edge $(i,j)$ of $\sigma_1$ with $i,j\in x_2$.
  \item $x_2$ is a block of $\sigma_2$ and $x_1$ is the minimal length edge $(i,j)$ of $\sigma_1$ with $i<\min(x_2)\leq\max(x_2)<j$ if such an edge exists; and $x_1=\emptyset$ otherwise.
 \end{enumerate}
\end{prop}
\begin{proof}
  This is a straightforward verification using the above two lemmas. For
  instance, let us prove one direction in the second case.

  Assume that $\pi_1\leq\pi_2$ and $x_2$ is an edge $(a,b)$ of
  $\sigma_2$. Clearly, we have $\sigma_1\leq\sigma_2$.  For two positive
  integers $i,j$ with $i<j$, we claim that $i\ssim{\pi_1} -j$ if and only if
  $i\ssim{\sigma_1} j$ and $i\leq a<b\leq j$.

  Let $i\ssim{\pi_1} -j$.  Then we have $i\ssim{\sigma_1} j$.  Since
  $\pi_1\leq\pi_2$, we have $i\ssim{\pi_2} -j$ and by Lemma~\ref{lem:t1} we get
  $i\leq a<b\leq j$.  Conversely, let $i\ssim{\sigma_1} j$ and $i\leq a<b\leq
  j$. Since $\sigma_1\leq\sigma_2$, we have $i\ssim{\sigma_2} j$ and by
  Lemma~\ref{lem:t2} we get $i\notssim{\pi_2} j$ and $i\ssim{\pi_2} -j$. On the
  other hand, $i\ssim{\sigma_1} j$ implies $i\ssim{\pi_1} j$ or $i\ssim{\pi_1}
  -j$. Since $\pi_1\leq\pi_2$, we cannot have $i\ssim{\pi_1} j$. Thus we get
  $i\ssim{\pi_1} -j$.

  Thus we have $i\ssim{\pi_1} -j$ if and only if $i\ssim{\sigma_1} j$ and $i\leq
  a<b\leq j$. By Lemma~\ref{lem:t1}, $x_1$ is the minimal length edge $(i,j)$ of
  $\sigma_1$ subject to $i\ssim{\pi_1} -j$, if such an edge exists; and
  $x_1=\emptyset$ otherwise. Thus we are done.

The other cases can be shown similarly.
\end{proof}

\section{$k$-divisible noncrossing partitions of type $A$}
\label{sec:k-divis-noncr}

Let $k$ be a positive integer.  A noncrossing partition $\pi\in\nca(kn)$ is
\emph{$k$-divisible} if the size of each block is divisible by $k$. Let
$\nck(n)$ denote the subposet of $\nca(kn)$ consisting of $k$-divisible
noncrossing partitions.  Then $\nck(n)$ is a graded poset with rank function
$\rank(\pi)=n-\bk(\pi)$, where $\bk(\pi)$ is the number of blocks of $\pi$.

In order to prove \eqref{eq:edelman}, Edelman \cite{Edelman1980} found a
bijection between the set of pairs $(\mcc, a)$ of a multichain $\mcc:\pi_1\leq
\pi_2\leq \cdots \leq \pi_{\ell+1}$ in $\nck(n)$ with rank jump vector
$(s_1,s_2,\ldots,s_{\ell+1})$ and an integer $a\in[n]$ and the set of
$(\ell+1)$-tuples $(L, R_1,R_2,\ldots,R_\ell)$ with $L\subset [n]$, $|L|=n-s_1$,
$R_i\subset [kn]$, and $|R_i|=s_{i+1}$ for $i\in[\ell]$. This bijection has been
extended to the noncrossing partitions of type $B_n$ \cite{Armstrong,
  Reiner1997} and type $D_n$ \cite{Athanasiadis2005}.

In this section we prove Theorem~\ref{thm:kratt1} by modifying the idea of Edelman.
Let us first introduce several notations. 

\subsection{The cyclic parenthesization}

Let $P(n)$ denote the set of pairs $(L,R)$ of subsets $L,R\subset [n]$ with the
same cardinality.  Let $(L,R)\in P(n)$. We can identify $(L,R)$ with the
\emph{cyclic parenthesization} of $(L,R)$ defined as follows.  We place a left
parenthesis before the occurrence of $i$ for each $i\in L$ and a right
parenthesis after the occurrence of $i$ for each $i\in R$ in the sequence
$1,2,\ldots,n$. We consider this sequence in cyclic order.

For  $x\in R$, the \emph{size of $x$} is defined to be the number of integers enclosed by $x$ and its corresponding left parenthesis, which are not enclosed by any other matching pair of parentheses. The \emph{type} of $(L,R)$, denoted by $\type(L,R)$, is defined to be $(b;b_1,b_2,\ldots,b_n)$, where $b_i$ is the number of $x\in R$ whose sizes are equal to $i$ and $b=b_1+b_2+\cdots+b_n$. 

\begin{example}\label{ex:1}
Let $(L,R)=(\{2,3,9,11,15,16\}, \{1,4,5,8,9,12\})\in P(16)$. Then the cyclic parenthesization is the following:
\begin{equation}
  \label{eq:16}
1) \:\: (2 \:\: (3 \:\: 4) \:\: 5) \:\: 6 \:\: 7 \:\: 8) \:\: (9) \:\: 10 \:\: (11 \:\: 12) \:\: 13 \:\: 14 \:\: (15 \:\: (16
\end{equation}
Since we consider \eqref{eq:16} in cyclic order, the right parenthesis of $1$ is matched with the left parenthesis of $16$ and the right parenthesis of $8$ is matched with the left parenthesis of $15$. The sizes of $5$ and $8$ in $R$ are $2$ and $4$ respectively. We have $\type(L,R)=(6;1,4,0,1,0,\ldots,0)$. 
\end{example}

Let $\ovp(n)$ denote the set of elements $(L,R)\in P(n)$ such that if
$\type(L,R)=(b;b_1,b_2,\ldots,b_n)$ then $\sum_{i=1}^n i b_i <n$. Thus we have
$(L,R)\in\ovp(n)$ if and only if there is at least one integer in the cyclic
parenthesization of $(L,R)$ which is not enclosed by any matching pair of
parentheses.

We define a map $\tau$ from $\ovp(n)$ to the set of pairs $(B,\pi)$, where
$\pi\in\nca(n)$ and $B$ is a block of $\pi$ as follows. Let $(L,R)\in
\ovp(n)$. Find a matching pair of parentheses in the cyclic parenthesization of
$(L,R)$ which do not enclose any other parenthesis. Remove the integers enclosed
by these parentheses, and make a block of $\pi$ with these integers. Repeat this
procedure until there is no parenthesis. Since $(L,R)\in\ovp(n)$, we have at
least one remaining integer after removing all the parentheses. These integers
also form a block of $\pi$ and $B$ is defined to be this block.

\begin{example}
  Let $(L,R)$ be the pair in Example~\ref{ex:1}. Note that
  $(L,R)\in\ovp(16)$. Then $\tau(L,R)=(B,\pi)$, where 
$$\pi = \{\{1,16\}, \{2,5\}, \{3,4\}, \{6,7,8,15\}, \{9\}, \{11,12\},
\{10,13,14\}\}$$ and $B=\{10,13,14\}$.
\end{example}

\begin{prop}
  The map $\tau$ is a bijection between $\ovp(n)$ and the set of pairs $(B,\pi)$, where $\pi\in\nca(n)$ and $B$ is a block of $\pi$. Moreover, if $\tau(L,R)=(B,\pi)$, $\type(\pi)=(b;b_1,b_2,\ldots,b_n)$ and $|B|=j$, then $\type(L,R)=(b-1,b_1,\ldots,b_j-1,\ldots,b_n)$. 
\end{prop}
\begin{proof}
Let $\tau(L,R)=(B,\pi)$. It is clear that $\pi\in\nca(n)$. It is sufficient to find the inverse map of $\tau$. 

Let $\pi\in\nca(n)$ and $B$ be a block of $\pi$. We will find $L$ and $R$ as follows. First we start with $L=R=\emptyset$ and the cyclic sequence $P$ which is initially $1,2,\ldots,n$. If $\pi$ has more than one block, then let $A$ be a block of $\pi$ not equal to $B$ such that the elements of $A$ are consecutive in the cyclic sequence $P$. Note that we can find such a block because $\pi$ is a noncrossing partition. We add the first and the last integers of $A$ in the cyclic sequence $P$ to $L$ and $R$ respectively. And then we remove the block $A$ from $\pi$ and remove the integers in $A$ in the cyclic sequence $P$. Repeat this procedure until $\pi$ has only one block $B$. Then the resulting sets $L$ and $R$ satisfy $\tau(L,R)=(B,\pi)$. Thus $\tau$ is a bijection.

The `moreover' statement is obvious from the construction of $\tau$.
\end{proof}

We define $P(n,\ell)$ to be the set of $(\ell+1)$-tuples $(L,R_1,R_2,\ldots,R_\ell)$ such that
$L,R_1,R_2,\ldots,R_\ell\subset[n]$ and $|L|=|R_1|+|R_2|+\cdots+|R_\ell|$. 
Similarly, we can consider the \emph{labeled cyclic parenthesization} of $(L,R_1,R_2,\ldots,R_\ell)$ by placing a left parenthesis before $i$ for each $i\in L$ and right parentheses $)_{j_1})_{j_2}\cdots )_{j_t}$ labeled with $j_1<j_2<\cdots<j_t$ after $i$ if $R_{j_1} ,R_{j_2} ,\ldots,R_{j_t}$ are the sets containing $i$ among $R_1,R_2,\ldots,R_\ell$. For each element $x\in R_i$, the \emph{size of $x$} is defined in the same way as in the case of $(L,R)$. We define the \emph{type} of $(L,R_1,R_2,\ldots,R_\ell)$ similarly to the type of $(L,R)$.

\begin{example}
Let $T=(L,R_1,R_2)=(\{2,4,5\}, \{2\}, \{2,6\})\in P(7,2)$. Then the labeled cyclic parenthesization of $T$ is the following:
\begin{equation}
  \label{eq:16b}
1 \:\: (2)_1)_2 \:\: 3 \:\: (4 \:\: (5 \:\: 6)_2 \:\: 7
\end{equation}
 Furthermore the size of $2\in R_1$ is $1$, the size of $2\in R_2$ is $3$ and the size of $6\in R_2$ is $2$. Thus the type of $T$ is $(3;1,1,1,0,\ldots,0)$.
\end{example}

\begin{lem}\label{thm:16}
Let $b,b_1,b_2,\ldots,b_n$ and $c_1,c_2,\ldots,c_\ell$ be nonnegative integers with $b=b_1+b_2+\cdots+b_n=c_1+c_2+\cdots+c_\ell$. Then
  the number of elements $(L,R_1,R_2,\ldots,R_\ell)$ in $P(n,\ell)$ with type $(b;b_1,b_2,\ldots,b_n)$ and $|R_i|=c_i$ for $i\in[\ell]$ is equal to
$$\binom{b}{b_1,b_2,\ldots,b_n}\binom{n}{c_1}\binom{n}{c_2}\cdots\binom{n}{c_\ell}.$$
\end{lem}
\begin{proof}
  Let $(L,R_1,R_2,\ldots,R_\ell)\in P(n,\ell)$ satisfy the conditions. Let $x_{i1},x_{i2},\ldots,x_{ic_i}$ be the elements of $R_i$ with $x_{i1}<x_{i2}<\cdots<x_{ic_i}$. Let $t_{ij}$ be the size of $x_{ij}\in R_i$. Then the sequence 
$$t_{11},\ldots, t_{1c_1}, t_{21},\ldots,t_{2c_2},\ldots, t_{\ell 1},\ldots,
t_{\ell c_\ell}$$
is an arrangement of  $\overbrace{1,\ldots,1}^{b_1},\overbrace{2,\ldots,2}^{b_2},\ldots, \overbrace{n,\ldots,n}^{b_n}$.
It is not difficult to see that $(L,R_1,R_2,\ldots,R_\ell)$ is completely
determined by the $R_i$'s and $t_{ij}$'s. Thus we are done.
\end{proof}

Let $\ovp(n,\ell)$ denote the set of $(L,R_1,R_2,\ldots,R_\ell)\in P(n,\ell)$ such that the type $(b;b_1,b_2,\ldots,b_n)$ of $(L,R_1,R_2,\ldots,R_\ell) $ satisfies $\sum_{i=1}^n i b_i <n$. 

Using $\tau$, we define a map $\tau'$ from $\ovp(n,\ell)$ to the set of pairs $(B,\mcc)$, where $\mcc:\pi_1\leq\pi_2\leq\cdots\leq\pi_\ell$ is a multichain in $\nca(n)$ and $B$ is a block of $\pi_1$ as follows. Let $P=(L,R_1,R_2,\ldots,R_\ell)\in\ovp(n,\ell)$. Applying the same argument as in the case of $\tau$ to the labeled cyclic parenthesization of $P$, we get $(B_1,\pi_1)$. Then remove all the right parentheses in $R_1$ from the cyclic parenthesization and their corresponding left parentheses. By repeating this procedure, we get $(B_i,\pi_i)$ for $i=2,3,\ldots,\ell$. Then we obtain a multichain $\mcc:\pi_1\leq\pi_2\leq\cdots\leq\pi_\ell$ in $\nca(n)$. We define $\tau'(P)=(B_1,\mcc)$.

\begin{prop}\label{thm:4}
  The map $\tau'$ is a bijection between $\ovp(n,\ell)$ and the set of pairs $(B,\mcc)$ where $\mcc:\pi_1\leq\pi_2\leq\cdots\leq\pi_\ell$ is a multichain in $\nca(n)$ and $B$ is a block of $\pi_1$. Moreover, if $\tau'(L,R_1,R_2,\ldots,R_\ell)=(B,\mcc)$, the rank jump vector of $\mcc$ is $(s_1,s_2,\ldots,s_{\ell+1})$, $\type(\pi_1)=(b;b_1,b_2,\ldots,b_n)$ and $|B|=j$, then the type of 
$(L,R_1,R_2,\ldots,R_\ell)$ is $(b-1;b_1,\ldots,b_j-1,\ldots,b_n)$ and $(|R_1|,|R_2|,\ldots,|R_\ell|)=(s_2,s_3,\ldots,s_{\ell+1})$.
\end{prop}
\begin{proof}
  We will find the inverse map of $\tau'$. This can be done by using a similar
  argument as in \cite{Athanasiadis2005,Edelman1980,Reiner1997}. Let
  $\mcc:\pi_1\leq\pi_2\leq\cdots\leq\pi_\ell$ be a multichain in $\nca(n)$ and
  $B$ be a block of $\pi_1$. Let $B_i$ be the block of $\pi_i$ containing
  $B$. Now we construct the corresponding labeled cyclic parenthesization as
  follows.

Let $P_\ell$ be the labeled cyclic parenthesization $\tau^{-1}(B_\ell,\pi_\ell)$ where all the right parentheses are labeled with $\ell$. Then for $i\in[\ell-1]$ we define $P_i$ to be the labeled cyclic parenthesization obtained from $P_{i+1}$ by adding the left parentheses, which are not already in $P_{i+1}$, and the corresponding right parentheses of $\tau^{-1}(B_i,\pi_i)$ and by labeling the newly added right parentheses with $i$. It is not difficult to check that the map $(B,\mcc)\mapsto P_1$ is the inverse map of $\tau'$.

The `moreover' statement is obvious from the construction of $\tau'$.
\end{proof}

\begin{thm}\label{thm:15}
Let $b,b_1,b_2,\ldots,b_n$ and $s_1,s_2,\ldots,s_{\ell+1}$ be nonnegative integers satisfying $\sum_{i=1}^n b_i=b$, $\sum_{i=1}^n i\cdot b_i=n$, $\sum_{i=1}^{\ell+1} s_i = n-1$ and $s_1=n-b$. Then the number of multichains $\pi_1\leq \pi_2\leq \cdots \leq \pi_{\ell}$ in $\nca(n)$ with rank jump vector $(s_1,s_2,\ldots,s_{\ell+1})$, $\type(\pi_1)=(b;b_1,b_2,\ldots,b_n)$ is equal to
$$\frac{1}{b} \binom{b}{b_1,b_2,\ldots,b_n} \binom{n}{s_2} \cdots \binom{n}{s_{\ell+1}}.$$
\end{thm}
\begin{proof}
By Lemma~\ref{thm:16} and Proposition~\ref{thm:4}, the number of pairs $(B,\mcc)$, where $\mcc$ is a multichain satisfying the conditions and $B$ is a block of $\pi_1$, is equal to
$$\sum_{j=1}^n \binom{b-1}{b_1,\ldots,b_j-1,\ldots,b_n} \binom{n}{s_2} \cdots \binom{n}{s_{\ell+1}} = \binom{b}{b_1,b_2,\ldots,b_n} \binom{n}{s_2} \cdots \binom{n}{s_{\ell+1}}.$$
Since there are $b=\bk(\pi_1)$ choices of $B$ for each $\mcc$, we get the theorem. Note that Lemma~\ref{thm:16} states the number of elements in $P(n,\ell)$. However, by the condition on the type, all the elements in consideration are in $\ovp(n,\ell)$.
\end{proof}

Now we can prove Theorem~\ref{thm:kratt1}.

\begin{proof}[Proof of Theorem~\ref{thm:kratt1}]
  Let $\pi_1\leq\pi_2\leq\cdots\leq\pi_\ell$ be a multichain in $\nck(n)$ with
  rank jump vector $(s_1,s_2,\ldots,s_{\ell+1})$ and
  $\typek(\pi_1)=(b;b_1,b_2,\ldots,b_n)$. Note that this is also a multichain in
  $\nca(kn)$ and $\rank(\pi_1)=kn-b$. Thus the rank jump vector of this
  multichain in $\nca(kn)$ is $(kn-b,s_2,\ldots,s_{\ell+1})$ and
  $\type(\pi_1)=(b; b_1',b_2',\ldots,b_{kn}')$ where $b_{ki}'=b_i$ for $i\in[n]$
  and $b_j'=0$ if $j$ is not divisible by $k$.  By Theorem~\ref{thm:15}, the
  number of such multichains is equal to
$$\frac{1}{b}\binom{b}{b_1',b_2',\ldots,b_{kn}'} \binom{kn}{s_2} \cdots \binom{kn}{s_{\ell+1}}=
\frac{1}{b}\binom{b}{b_1,b_2,\ldots,b_n} \binom{kn}{s_2} \cdots \binom{kn}{s_{\ell+1}}.$$
\end{proof}

\subsection{Augmented $k$-divisible noncrossing partitions of type $A$}

If all the block sizes of a partition $\pi\in\Pi(n)$ are divisible by $k$ then
$n$ must be divisible by $k$. Thus $k$-divisible partitions can be defined only
on $\Pi(kn)$. We extend this definition to partitions in $\Pi(n)$ for integers
$n$ not divisible by $k$ as follows.

Let $k$ and $r$ be integers with $0<r<k$.  A partition $\pi$ of $[kn+r]$ is \emph{augmented $k$-divisible} if the sizes of all but one of the blocks are divisible by $k$. 

Let $\nck(n;r)$ denote the subposet of $\nca(kn+r)$ consisting of the augmented $k$-divisible noncrossing partitions.  Then $\nck(n;r)$ is a graded poset with rank function $\rank(\pi)=n-\bk'(\pi)$, where $\bk'(\pi)$ is the number of blocks of $\pi$ whose sizes are divisible by $k$. We define $\typek(\pi)$ to be $(b;b_1,b_2,\ldots,b_n)$ where $b_i$ is the number of blocks $B$ of size $ki$ and $b=b_1+b_2+\cdots+b_n$. 

\begin{thm}\label{thm:6}
Let $0<r<k$. 
Let $b,b_1,b_2,\ldots,b_n$ and $s_1,s_2,\ldots,s_{\ell+1}$ be nonnegative integers satisfying $\sum_{i=1}^n b_i=b$, $\sum_{i=1}^n i\cdot b_i\leq n$, $\sum_{i=1}^{\ell+1} s_i = n$ and $s_1=n-b$. 
Then the number of multichains $\mcc:\pi_1\leq \pi_2\leq \cdots \leq \pi_{\ell}$ in $\nck(n;r)$ with rank jump vector $(s_1,s_2,\ldots,s_{\ell+1})$ and $\typek(\pi_1)=(b;b_1,b_2,\ldots,b_n)$ is equal to
$$\binom{b}{b_1,b_2,\ldots,b_n} \binom{kn+r}{s_2} \cdots \binom{kn+r}{s_{\ell+1}}.$$
\end{thm}
\begin{proof}
  Let $m=kn+r-\sum_{i=1}^n ki\cdot b_{i}$.  Observe that, as an element of
  $\nca(kn+r)$, $\pi_1$ has type $(b+1;b_1',b_2',\ldots,b_{kn+r}')$, where
  $b_{ki}'=b_i$, $b_m'=1$ and $b_j'=0$ if $j\ne m$ and $j\ne ki$ for any
  $i\in[n]$. For a given multichain $\mcc$ satisfying the conditions, let $B$ be
  the unique block of $\pi_1$ whose size is $m$. By Lemma~\ref{thm:16} and
  Proposition~\ref{thm:4}, the number of pairs $(B,\mcc)$ is equal to
$$\binom{b}{b_1',\ldots,b_m'-1,\ldots,b_{kn+r}'} \binom{kn+r}{s_2} \cdots \binom{kn+r}{s_{\ell+1}},$$
which is equal to the desired formula.
\end{proof}

It is not difficult to see that for fixed integers $n$ and $b$, the sum of
$\binom{b}{b_1,b_2,\ldots,b_n}$ over all $n$-tuples $(b_1,b_2,\ldots,b_n)$ of
nonnegative integers such that $\sum_{i=1}^n b_i = b$ and $\sum_{i=1}^n i b_i
\leq n$ is equal to $\binom{n}{b}$.  Thus adding up the expression in
Theorem~\ref{thm:6} for all possible nonnegative integers $b_1,b_2,\ldots,b_n$,
we get the following corollary.

\begin{cor}\label{thm:1}
Let $0<r<k$. 
Let $s_1,s_2,\ldots,s_{\ell+1}$ be nonnegative integers with $s_1+\cdots+s_{\ell+1}=n$. Then the number of multichains in $\nck(n;r)$ with rank jump vector $(s_1, s_2,\ldots, s_{\ell+1})$ is equal to
  $$\binom{n}{s_1}\binom{kn+r}{s_2}\cdots\binom{kn+r}{s_{\ell+1}}.$$
\end{cor}

The \emph{zeta polynomial} $Z(P,\ell)$ of a poset $P$ is the number of
multichains $\pi_1\leq \pi_2\leq\cdots\leq \pi_\ell$ in $P$.  By summing the formula
in Corollary~\ref{thm:1} for all $(\ell+1)$-tuples $(s_1,\ldots,s_{\ell+1})$
with $s_1+\cdots+s_{\ell+1}=n$, we get the zeta polynomial of $\nck(n;r)$.

\begin{cor}\label{thm:2}
Let $0<r<k$.   The zeta polynomial of $\nck(n;r)$ is given by
$$Z(\nck(n;r),\ell)=\binom{n+\ell (kn+r)}{n}.$$
\end{cor}

Corollary~\ref{thm:2} will be used to prove Armstrong's conjecture in
Section~\ref{sec:armstr-conj}.

\section{$k$-divisible noncrossing partitions of type $B$}
\label{sec:kb}

Let $\pi\in\ncb(kn)$. We say that $\pi$ is a \emph{$k$-divisible noncrossing partition of type $B_n$} if the size of each block of $\pi$ is divisible by $k$. 

Let $\nckb(n)$ denote the subposet of $\ncb(kn)$ consisting of $k$-divisible noncrossing partitions of type $B_n$. Then $\nckb(n)$ is a graded poset with rank function $\rank(\pi)=n-\nz(\pi)$, where $\nz(\pi)$ denotes the number of unordered pairs $(B,-B)$ of nonzero blocks of $\pi$. 

We can prove Theorem~\ref{thm:kratt2} using a similar method as in the proof of
Theorem~\ref{thm:kratt1}. Instead of doing this, we will prove the following
lemma which implies Theorem~\ref{thm:kratt2}. Note that the following lemma is
also a generalization of \eqref{eq:14}.

For a multichain $\mcc:\pi_1\leq\pi_2\leq\cdots\leq\pi_\ell$ in $\nckb(n)$, the
\emph{index} $\ind(\mcc)$ of $\mcc$ is the smallest integer $d$ such that
$\pi_d$ has a zero block. If there is no such integer $d$, then
$\ind(\mcc)=\ell+1$.

\begin{lem}\label{thm:atha}
Let $b,b_1,b_2,\ldots,b_n$ and $s_1,s_2,\ldots,s_{\ell+1}$ be nonnegative integers satisfying $\sum_{i=1}^n b_i=b$, $\sum_{i=1}^n i\cdot b_i\leq n$, $\sum_{i=1}^{\ell+1} s_i = n$ and $s_1=n-b$. Then the number of multichains $\mcc:\pi_1\leq \pi_2\leq\cdots\leq\pi_{\ell}$ in $\nckb(n)$ with rank jump vector $(s_1,s_2,\ldots,s_{\ell+1})$, $\typek(\pi_1)=(b;b_1,b_2,\ldots,b_n)$ and $\ind(\mcc)=d$ is equal to
$$\binom{b}{b_1,b_2,\ldots,b_n} \binom{kn}{s_2} \cdots \binom{kn}{s_{\ell+1}},$$
if $d=1$, and
$$\frac{s_d}{b} \binom{b}{b_1,b_2,\ldots,b_n} \binom{kn}{s_2} \cdots \binom{kn}{s_{\ell+1}},$$
if $d\geq2$.
\end{lem}
\begin{proof}
  Recall the bijection $\psi$ in Section~\ref{sec:interpr-noncr-part-2}.  Let
  $\psi(\pi_i)=(\sigma_i,x_i)$ for $i\in[\ell]$. By Proposition~\ref{thm:12}, we
  have $\sigma_i\in\nck(n)$ and
  $\sigma_1\leq\sigma_2\leq\cdots\leq\sigma_{\ell}$. Since $\ind(\mcc)=d$, $x_d$
  is a block of $\sigma_d$ and $x_{d-1}$ is not a block of $\sigma_{d-1}$. By
  Proposition~\ref{thm:12}, $x_1,x_2,\ldots,x_{d-2}$ (resp.~$x_{d+1},\ldots,x_{\ell}$)
  are completely determined by $x_{d-1}$ (resp.~$x_d$) together with
  $\sigma_1,\sigma_2,\ldots,\sigma_{\ell}$.  Moreover, $x_{d-1}$ is one of the
  following:
\begin{itemize}
\item an edge $(u,v)$ of $\sigma_{d-1}$ such that $u,v\in x_d$,
\item the minimal length edge $(u,v)$ of $\sigma_{d-1}$ with $u<\min(x_d)\leq\max(x_d)<v$ if such an edge exists; and $\emptyset$ otherwise.
\end{itemize}
Again by Proposition~\ref{thm:12}, for any choice of $\mcc':\sigma_1\leq\sigma_2\leq\cdots\leq\sigma_{\ell}$, $x_d$ and $x_{d-1}$ with the above conditions, we can construct the corresponding multichain $\pi_1\leq \pi_2\leq\cdots\leq\pi_{\ell}$. 

Assume $d=1$. Then $\pi_1$ has a zero block. Let $2t\cdot k$ be the size of the zero block of $\pi_1$. Then $\type(\sigma_1)=(b+1;b_1,\ldots,b_t+1,\ldots,b_n)$ and the rank jump vector of $\mcc'$ is $(s_1-1,s_2,\ldots,s_{\ell+1})$. Thus the number of such multichains $\mcc'$ is equal to 
$$\frac{1}{b+1} \binom{b+1}{b_1,\ldots, b_t+1, \ldots, b_n} \binom{kn}{s_2} \cdots \binom{kn}{s_{\ell+1}},$$
and the number of choices of $x_1$ is $b_t +1$. By multiplying these two numbers, we get the formula for the case $d=1$.

Now assume $d\geq2$. Then $\type(\sigma_1)=(b;b_1,b_2,\ldots,b_n)$ and the rank jump vector of $\mcc'$ is $(s_1,\ldots,s_d-1,\ldots,s_{\ell+1})$. Thus the number of such multichains $\mcc'$ is equal to 
\begin{equation}
  \label{eq:5}
\frac{1}{b} \binom{b}{b_1,b_2,\ldots,b_n} \binom{kn}{s_2} \cdots\binom{kn}{s_d-1}\cdots \binom{kn}{s_{\ell+1}}.
\end{equation}
Let us fix such a multichain $\mcc'$. Note that $x_d$ can be any block of
$\sigma_d$. If $x_d$ is a block $B$ of $\sigma_d$, then there are $1+e(B)$
choices of $x_{d-1}$, where $e(B)$ is the number of edges $(u,v)$ in
$\sigma_{d-1}$ with $u,v\in B$. Thus there are $\sum_{B\in\sigma_d} (1+e(B))$
choices for $x_d$ and $x_{d-1}$. Note that $B$ is a union of several blocks of
$\sigma_{d-1}$. Let $t(B)$ be the number of blocks of $\sigma_{d-1}$ whose union
is $B$. It is easy to see that $e(B)=|B|-t(B)$ for each block $B$ of
$\sigma_{d}$ and $\rank(\sigma_d)-\rank(\sigma_{d-1})=s_d-1=\sum_{B\in\sigma_d}
(t(B)-1)$.  Thus the number of choices of $x_d$ and $x_{d-1}$ is equal to
\begin{align*}
  \sum_{B\in\sigma_d} (1+e(B)) &=   \sum_{B\in\sigma_d} (1+|B|-t(B)) \\
  &= kn - \sum_{B\in\sigma_d} (t(B)-1) \\
 &= kn - (s_d-1).
\end{align*}
The product of this number and \eqref{eq:5} is equal to the formula for the case $d\geq2$.
\end{proof}

\section{Armstrong's conjecture}
\label{sec:armstr-conj}

Let $\tnc{k}{n}$ denote the subposet of $\nck(n)$ whose elements are fixed under
a $180^\circ$ rotation in the circular representation.

Armstrong \cite[Conjecture 4.5.14]{Armstrong} conjectured the following. Let $n$
and $k$ be integers such that $n$ is even and $k$ is arbitrary, or $n$ is odd
and $k$ is even. Then the zeta polynomial of $\tnc{k}{n}$ is given by
  $$Z(\tnc{k}{n},\ell) = \binom{\floor{(k\ell+1)n/2}}{\floor{n/2}}.$$

If $n$ is even then $\tnc{k}{n}$ is isomorphic to $\nckb(n/2)$, whose zeta polynomial is already known. If both $n$ and $k$ are odd, then $\tnc{k}{n}$ is empty. Thus the conjecture is only for $n$ and $k$ such that $n$ is odd and $k$ is even. In this case, the formula is equivalent to the following:
  $$Z(\tnc{2k}{2n+1},\ell) = \binom{n+\ell(2kn+k)}{n}.$$

\begin{thm}\label{thm:14}
Let $n$ and $k$ be positive integers. Then
  $$\tnc{2k}{2n+1}\cong\nctk(n;k).$$
\end{thm}
\begin{proof}
  Let $\pi\in \tnc{2k}{2n+1}$. Since $\pi$ is invariant under a $180^\circ$
  rotation, we can consider $\pi$ as an element of $\ncb(2kn+k)$.  Note that the
  sum of the block sizes of $\pi$ is divisible by $2k$ but not by $4k$.  Since
  all the block sizes are divisible by $2k$ and the nonzero blocks appear in
  pairs, $\pi$ must have a zero block of size divisible by $2k$ but not by
  $4k$. Let $B$ be the zero block of $\pi$. 

  Now consider the map $\psi:\ncb(n)\to\ncbb(n)$ defined in
  Section~\ref{sec:interpr-noncr-part-2}. Let $\psi(\pi)=(\sigma,x)$. Since
  $\pi$ has zero block $B$, by Theorem~\ref{thm:simbij}, $x$ is a block whose
  size is divisible by $k$ but not by $2k$. All the other blocks in $\sigma$
  have sizes divisible by $2k$. Thus $\sigma\in\nctk(n;k)$.  Since $x$ is the
  only block of $\sigma$ whose size is not divisible by $2k$, the map
  $\pi\mapsto \sigma$ is a bijection between $\tnc{2k}{2n+1}$ and
  $\nctk(n;k)$. Moreover, if $\pi<\pi'$ and $\pi'\mapsto \sigma'$, then
  $\sigma<\sigma'$. Thus this map is in fact a poset isomorphism.
\end{proof}

Let $\pi\in \tnc{2k}{2n+1}$. Since we can consider $\pi$ as an element in
$\nckb(2n+1)$, the $k$-type $\typek(\pi)$ is defined.  In other words,
$\typek(\pi)=(b;b_1,b_2,\ldots,b_n)$, where $2b_i$ is the number of blocks of
$\pi$ of size $ki$ which are not invariant under a $180^\circ$ rotation and
$b=b_1+b_2+\cdots+b_n$. 
Note that the map $\pi\mapsto\sigma$ in the proof of Theorem~\ref{thm:14}
satisfies $\typek(\pi)=\typek(\sigma)$.
 
By Theorem~\ref{thm:14}, we get the following results immediately from
Theorem~\ref{thm:6}, Corollaries~\ref{thm:1} and \ref{thm:2}.

\begin{thm}
Let $b,b_1,b_2,\ldots,b_n$ and $s_1,s_2,\ldots,s_{\ell+1}$ be nonnegative integers satisfying $\sum_{i=1}^n b_i=b$, $\sum_{i=1}^n i\cdot b_i\leq n$, $\sum_{i=1}^{\ell+1} s_i = n$ and $s_1=n-b$. Then the number of multichains $\pi_1\leq \pi_2\leq \cdots \leq \pi_{\ell}$ in $\tnc{2k}{2n+1}$ with rank jump vector $(s_1,s_2,\ldots,s_{\ell+1})$ and $\typek(\pi_1)=(b;b_1,b_2,\ldots,b_n)$ is equal to
$$\binom{b}{b_1,b_2,\ldots,b_n} \binom{2kn+k}{s_2} \cdots \binom{2kn+k}{s_{\ell+1}}.$$
\end{thm}

\begin{cor}
Let $s_1,s_2,\ldots,s_{\ell+1}$ be nonnegative integers with $s_1+s_2+\cdots+s_{\ell+1}=n$. Then the number of multichains in $\tnc{2k}{2n+1}$ with rank jump vector $(s_1, s_2,\ldots, s_{\ell+1})$ is equal to
  $$\binom{n}{s_1}\binom{2kn+k}{s_2}\cdots\binom{2kn+k}{s_{\ell+1}}.$$
\end{cor}

\begin{cor}\cite[Conjecture 4.5.14]{Armstrong}
  The zeta polynomial of $\tnc{2k}{2n+1}$ is given by
$$Z(\tnc{2k}{2n+1},\ell)=\binom{n+\ell (2kn+k)}{n}.$$
\end{cor}

\section{$k$-divisible noncrossing partitions of type $D$}
\label{sec:kd}

Krattenthaler and M{\"u}ller \cite{KrattMuller} found a combinatorial realization of $k$-divisible noncrossing partitions of type $D_n$. Let us review their definition. 

Let $\pi$ be a partition of $\{1,2,\ldots,kn,-1,-2,\ldots,-kn\}$. We represent $\pi$ using the annulus with the integers $1,2,\ldots,kn-k, -1,-2,\ldots,-(kn-k)$ on the outer circle in clockwise order and the integers $kn-k+1,kn-k+2,\ldots,kn, -(kn-k+1), -(kn-k+2), \ldots,-kn$ on the inner circle in counterclockwise order. We represent each block $B$ of $\pi$ as follows. Let $B=\{a_1,a_2,\ldots,a_u\}$. We can assume that $a_1,a_2,\ldots, a_i$ are arranged on the outer circle in clockwise order and $a_{i+1},a_{i+2},\ldots,a_u$ are arranged on the inner circle in counterclockwise order for some $i$. Then we draw curves, which lie inside of the annulus, connecting $a_j$ and $a_{j+1}$ for all $j\in[u]$, where $a_{u+1}=a_1$. If we can draw the curves for all the blocks of $\pi$ such that they do not intersect, then we call $\pi$ a \emph{noncrossing partition on the $(2k(n-1), 2k)$-annulus}.

A \emph{$k$-divisible noncrossing partition of type $D_n$} is a noncrossing partition $\pi$ on the $(2k(n-1), 2k)$-annulus satisfying the following conditions:
\begin{enumerate}
\item If $B\in\pi$, then $-B\in\pi$.
\item For each block $B\in\pi$ with $u$ elements, if $a_1,a_2,\ldots,a_u$ are the elements of $B$ such that $a_1,a_2,\ldots, a_i$ are arranged on the outer circle in clockwise order and $a_{i+1},a_{i+2},\ldots,a_u$ are arranged on the inner circle in counterclockwise order, then $|a_{j+1}| \equiv |a_j| +1 \mod k$ for all $j\in[u]$, where $a_{u+1}=a_1$.
\item If there is a zero block $B$ of $\pi$, i.e.~$B=-B$, then $B$ contains all
  the integers on the inner circle and at least two integers on the outer circle.
\item If every block contains integers entirely on the inner circle or on the
  outer circle, then the blocks on the inner circle are completely determined by
  the blocks on the outer circle in the following way.  A block on the outer
  circle is \emph{visible} if there is no block between this block and the inner
  circle, in other words, we can connect this block and the inner circle using a
  curve lying inside of the annulus without touching any other block. Take a
  visible block $A$ and assume that $a$ is the last element in $A$ when we read
  the elements in $A$ in clockwise order. More precisely, if we read the
  elements in $A\cup(-A)$ in clockwise order, then $a$ is followed by an element
  in $-A$.  There is a unique positive integer $b$ on the inner circle with
  $a\equiv b \mod k$. In this case the integers on the inner circle are
  partitioned into two blocks $B$ and $-B$, where the elements of $B$ are the
  $k$ consecutive integers on the inner circle ending with $b$ in
  counterclockwise order. Note that $B$ is independent of the choice of $A$.
\end{enumerate}

The fourth condition is natural in the sense that it is the only way of
partitioning the integers on the inner circle into two blocks in such a way that
we can connect visible blocks and blocks on the inner circle so that the
resulting partition is still in $\nckd(n)$. We note that the fourth condition is
mistakenly missing in \cite{KrattMuller} (private communication with Christian
Krattenthaler).

Let $\nckd(n)$ denote the set of $k$-divisible noncrossing partitions of type $D_n$. 
See Figures~\ref{fig:annulus1} and \ref{fig:annulus2} for some elements in $\nckd(n)$. 

\begin{figure}
  \centering
\centering
\scriptsize
\begin{pspicture}(-6,-5.5)(6,5.5) 
\pscircle[fillstyle=solid](0,0){2.50000000000000}
\pscircle(0,0){5}
\pspolygon[fillstyle=solid,fillcolor=gray](3.172,-3.865)(2.357,-4.410)(1.451,-4.785)(0.4901,-4.976)
\pspolygon[fillstyle=solid,fillcolor=gray](-3.172,3.865)(-2.357,4.410)(-1.451,4.785)(-0.4901,4.976)
\pscircle[fillstyle=solid](0,0){2.50000000000000}
\pscircle(0,0){5}
\pspolygon[fillstyle=solid,fillcolor=gray](0.4901,4.976)(1.451,4.785)(2.357,4.410)(3.172,3.865)
\pspolygon[fillstyle=solid,fillcolor=gray](-0.4899,-4.976)(-1.452,-4.785)(-2.357,-4.410)(-3.172,-3.865)
\pscircle[fillstyle=solid](0,0){2.50000000000000}
\pscircle(0,0){5}
\pspolygon[fillstyle=solid,fillcolor=gray](3.865,3.172)(4.410,2.357)(4.410,-2.357)(3.865,-3.172)
\pspolygon[fillstyle=solid,fillcolor=gray](-3.865,-3.172)(-4.410,-2.357)(-4.410,2.357)(-3.865,3.172)
\pscircle[fillstyle=solid](0,0){2.50000000000000}
\pscircle(0,0){5}
\pspolygon[fillstyle=solid,fillcolor=gray](4.785,1.451)(4.976,0.4901)(4.976,-0.4901)(4.785,-1.451)
\pspolygon[fillstyle=solid,fillcolor=gray](-4.785,-1.451)(-4.976,-0.4901)(-4.976,0.4900)(-4.785,1.451)
\cput{0.4901,4.976}{84.38}{1}{1}
\cput{-0.4899,-4.976}{264.4}{-1}{-1}
\cput{1.451,4.785}{73.12}{2}{2}
\cput{-1.452,-4.785}{253.1}{-2}{-2}
\cput{2.357,4.410}{61.88}{3}{3}
\cput{-2.357,-4.410}{241.9}{-3}{-3}
\cput{3.172,3.865}{50.62}{4}{4}
\cput{-3.172,-3.865}{230.6}{-4}{-4}
\cput{3.865,3.172}{39.38}{5}{5}
\cput{-3.865,-3.172}{219.4}{-5}{-5}
\cput{4.410,2.357}{28.12}{6}{6}
\cput{-4.410,-2.357}{208.1}{-6}{-6}
\cput{4.785,1.451}{16.88}{7}{7}
\cput{-4.785,-1.451}{196.9}{-7}{-7}
\cput{4.976,0.4901}{5.625}{8}{8}
\cput{-4.976,-0.4901}{185.6}{-8}{-8}
\cput{4.976,-0.4901}{-5.625}{9}{9}
\cput{-4.976,0.4900}{174.4}{-9}{-9}
\cput{4.785,-1.451}{-16.88}{10}{10}
\cput{-4.785,1.451}{163.1}{-10}{-10}
\cput{4.410,-2.357}{-28.12}{11}{11}
\cput{-4.410,2.357}{151.9}{-11}{-11}
\cput{3.865,-3.172}{-39.38}{12}{12}
\cput{-3.865,3.172}{140.6}{-12}{-12}
\cput{3.172,-3.865}{-50.62}{13}{13}
\cput{-3.172,3.865}{129.4}{-13}{-13}
\cput{2.357,-4.410}{-61.88}{14}{14}
\cput{-2.357,4.410}{118.1}{-14}{-14}
\cput{1.451,-4.785}{-73.12}{15}{15}
\cput{-1.451,4.785}{106.9}{-15}{-15}
\cput{0.4901,-4.976}{-84.38}{16}{16}
\cput{-0.4901,4.976}{95.62}{-16}{-16}
\cput{-0.9567,2.310}{292.5}{17}{17}
\cput{0.9567,-2.310}{472.5}{-17}{-17}
\cput{-2.310,0.9567}{337.5}{18}{18}
\cput{2.310,-0.9567}{517.5}{-18}{-18}
\cput{-2.310,-0.9567}{382.5}{19}{19}
\cput{2.310,0.9568}{562.5}{-19}{-19}
\cput{-0.9567,-2.310}{427.5}{20}{20}
\cput{0.9567,2.310}{607.5}{-20}{-20}
\end{pspicture}  
\begin{pspicture}(-6,-5.5)(6,5.5) 
\pscircle[fillstyle=solid](0,0){2.50000000000000}
\pscircle(0,0){5}
\pspolygon[fillstyle=solid,fillcolor=gray](3.172,-3.865)(2.357,-4.410)(1.451,-4.785)(0.4901,-4.976)(-0.4899,-4.976)(-1.452,-4.785)(-2.357,-4.410)(-3.172,-3.865)
\pspolygon[fillstyle=solid,fillcolor=gray](-3.172,3.865)(-2.357,4.410)(-1.451,4.785)(-0.4901,4.976)(0.4901,4.976)(1.451,4.785)(2.357,4.410)(3.172,3.865)
\pscircle[fillstyle=solid](0,0){2.50000000000000}
\pscircle(0,0){5}
\pspolygon[fillstyle=solid,fillcolor=gray](3.865,3.172)(4.410,2.357)(4.410,-2.357)(3.865,-3.172)(0.9567,-2.310)(2.310,-0.9567)(2.310,0.9568)(0.9567,2.310)
\pspolygon[fillstyle=solid,fillcolor=gray](-3.865,-3.172)(-4.410,-2.357)(-4.410,2.357)(-3.865,3.172)(-0.9567,2.310)(-2.310,0.9567)(-2.310,-0.9567)(-0.9567,-2.310)
\pscircle[fillstyle=solid](0,0){2.50000000000000}
\pscircle(0,0){5}
\pspolygon[fillstyle=solid,fillcolor=gray](4.785,1.451)(4.976,0.4901)(4.976,-0.4901)(4.785,-1.451)
\pspolygon[fillstyle=solid,fillcolor=gray](-4.785,-1.451)(-4.976,-0.4901)(-4.976,0.4900)(-4.785,1.451)
\cput{0.4901,4.976}{84.38}{1}{1}
\cput{-0.4899,-4.976}{264.4}{-1}{-1}
\cput{1.451,4.785}{73.12}{2}{2}
\cput{-1.452,-4.785}{253.1}{-2}{-2}
\cput{2.357,4.410}{61.88}{3}{3}
\cput{-2.357,-4.410}{241.9}{-3}{-3}
\cput{3.172,3.865}{50.62}{4}{4}
\cput{-3.172,-3.865}{230.6}{-4}{-4}
\cput{3.865,3.172}{39.38}{5}{5}
\cput{-3.865,-3.172}{219.4}{-5}{-5}
\cput{4.410,2.357}{28.12}{6}{6}
\cput{-4.410,-2.357}{208.1}{-6}{-6}
\cput{4.785,1.451}{16.88}{7}{7}
\cput{-4.785,-1.451}{196.9}{-7}{-7}
\cput{4.976,0.4901}{5.625}{8}{8}
\cput{-4.976,-0.4901}{185.6}{-8}{-8}
\cput{4.976,-0.4901}{-5.625}{9}{9}
\cput{-4.976,0.4900}{174.4}{-9}{-9}
\cput{4.785,-1.451}{-16.88}{10}{10}
\cput{-4.785,1.451}{163.1}{-10}{-10}
\cput{4.410,-2.357}{-28.12}{11}{11}
\cput{-4.410,2.357}{151.9}{-11}{-11}
\cput{3.865,-3.172}{-39.38}{12}{12}
\cput{-3.865,3.172}{140.6}{-12}{-12}
\cput{3.172,-3.865}{-50.62}{13}{13}
\cput{-3.172,3.865}{129.4}{-13}{-13}
\cput{2.357,-4.410}{-61.88}{14}{14}
\cput{-2.357,4.410}{118.1}{-14}{-14}
\cput{1.451,-4.785}{-73.12}{15}{15}
\cput{-1.451,4.785}{106.9}{-15}{-15}
\cput{0.4901,-4.976}{-84.38}{16}{16}
\cput{-0.4901,4.976}{95.62}{-16}{-16}
\cput{-0.9567,2.310}{292.5}{17}{17}
\cput{0.9567,-2.310}{472.5}{-17}{-17}
\cput{-2.310,0.9567}{337.5}{18}{18}
\cput{2.310,-0.9567}{517.5}{-18}{-18}
\cput{-2.310,-0.9567}{382.5}{19}{19}
\cput{2.310,0.9568}{562.5}{-19}{-19}
\cput{-0.9567,-2.310}{427.5}{20}{20}
\cput{0.9567,2.310}{607.5}{-20}{-20}
\end{pspicture}  
 \caption{Two elements in $\NC_D^{(4)}(5)$. In the left diagram, the integers on the inner circle form two blocks $\pm\{17,18,19,20\}$.}
  \label{fig:annulus1}
\end{figure}
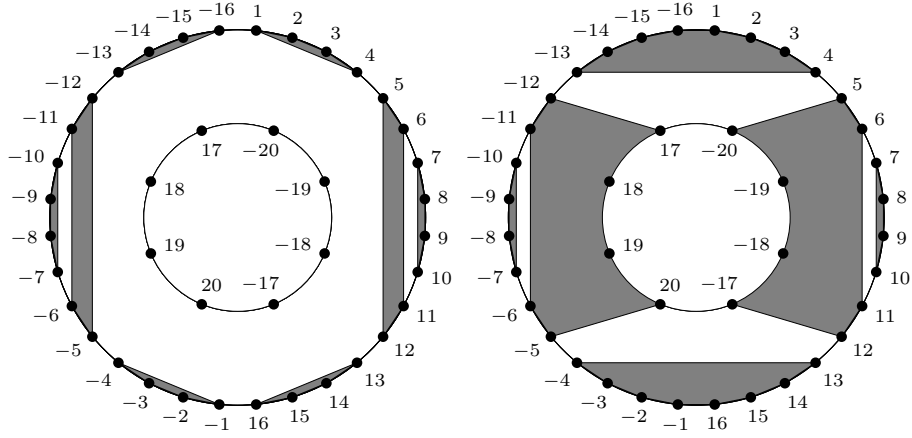

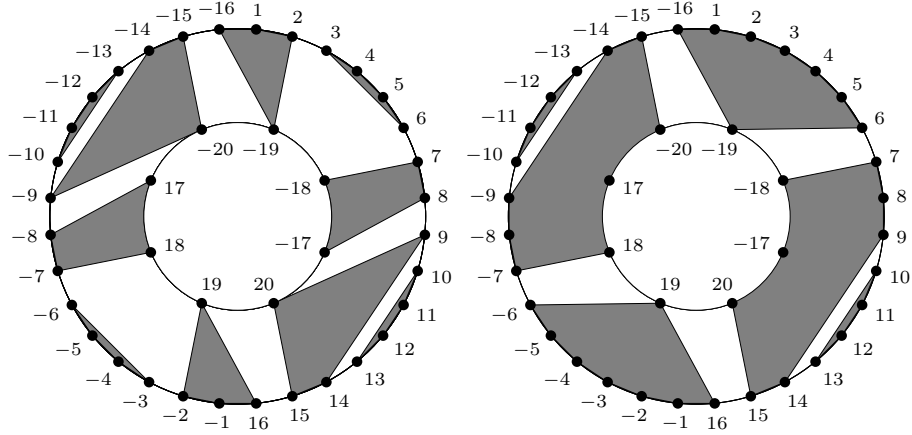
\begin{figure}
  \centering \scriptsize
\begin{pspicture}(-6,-5.5)(6,5.5) 
\pscircle[fillstyle=solid](0,0){2.50000000000000}
\pscircle(0,0){5}
\pspolygon[fillstyle=solid,fillcolor=gray](-0.4901,4.976)(0.4901,4.976)(1.451,4.785)(0.9567,2.310)
\pspolygon[fillstyle=solid,fillcolor=gray](0.4901,-4.976)(-0.4899,-4.976)(-1.452,-4.785)(-0.9567,-2.310)
\pscircle[fillstyle=solid](0,0){2.50000000000000}
\pscircle(0,0){5}
\pspolygon[fillstyle=solid,fillcolor=gray](2.357,4.410)(3.172,3.865)(3.865,3.172)(4.410,2.357)
\pspolygon[fillstyle=solid,fillcolor=gray](-2.357,-4.410)(-3.172,-3.865)(-3.865,-3.172)(-4.410,-2.357)
\pscircle[fillstyle=solid](0,0){2.50000000000000}
\pscircle(0,0){5}
\pspolygon[fillstyle=solid,fillcolor=gray](4.785,1.451)(4.976,0.4901)(2.310,-0.9567)(2.310,0.9568)
\pspolygon[fillstyle=solid,fillcolor=gray](-4.785,-1.451)(-4.976,-0.4901)(-2.310,0.9567)(-2.310,-0.9567)
\pscircle[fillstyle=solid](0,0){2.50000000000000}
\pscircle(0,0){5}
\pspolygon[fillstyle=solid,fillcolor=gray](4.976,-0.4901)(2.357,-4.410)(1.451,-4.785)(0.9567,-2.310)
\pspolygon[fillstyle=solid,fillcolor=gray](-4.976,0.4900)(-2.357,4.410)(-1.451,4.785)(-0.9568,2.310)
\pscircle[fillstyle=solid](0,0){2.50000000000000}
\pscircle(0,0){5}
\pspolygon[fillstyle=solid,fillcolor=gray](4.785,-1.451)(4.410,-2.357)(3.865,-3.172)(3.172,-3.865)
\pspolygon[fillstyle=solid,fillcolor=gray](-4.785,1.451)(-4.410,2.357)(-3.865,3.172)(-3.172,3.865)
\cput{0.4901,4.976}{84.38}{1}{1}
\cput{-0.4899,-4.976}{264.4}{-1}{-1}
\cput{1.451,4.785}{73.12}{2}{2}
\cput{-1.452,-4.785}{253.1}{-2}{-2}
\cput{2.357,4.410}{61.88}{3}{3}
\cput{-2.357,-4.410}{241.9}{-3}{-3}
\cput{3.172,3.865}{50.62}{4}{4}
\cput{-3.172,-3.865}{230.6}{-4}{-4}
\cput{3.865,3.172}{39.38}{5}{5}
\cput{-3.865,-3.172}{219.4}{-5}{-5}
\cput{4.410,2.357}{28.12}{6}{6}
\cput{-4.410,-2.357}{208.1}{-6}{-6}
\cput{4.785,1.451}{16.88}{7}{7}
\cput{-4.785,-1.451}{196.9}{-7}{-7}
\cput{4.976,0.4901}{5.625}{8}{8}
\cput{-4.976,-0.4901}{185.6}{-8}{-8}
\cput{4.976,-0.4901}{-5.625}{9}{9}
\cput{-4.976,0.4900}{174.4}{-9}{-9}
\cput{4.785,-1.451}{-16.88}{10}{10}
\cput{-4.785,1.451}{163.1}{-10}{-10}
\cput{4.410,-2.357}{-28.12}{11}{11}
\cput{-4.410,2.357}{151.9}{-11}{-11}
\cput{3.865,-3.172}{-39.38}{12}{12}
\cput{-3.865,3.172}{140.6}{-12}{-12}
\cput{3.172,-3.865}{-50.62}{13}{13}
\cput{-3.172,3.865}{129.4}{-13}{-13}
\cput{2.357,-4.410}{-61.88}{14}{14}
\cput{-2.357,4.410}{118.1}{-14}{-14}
\cput{1.451,-4.785}{-73.12}{15}{15}
\cput{-1.451,4.785}{106.9}{-15}{-15}
\cput{0.4901,-4.976}{-84.38}{16}{16}
\cput{-0.4901,4.976}{95.62}{-16}{-16}
\cput{-2.310,0.9567}{337.5}{17}{17}
\cput{2.310,-0.9567}{517.5}{-17}{-17}
\cput{-2.310,-0.9567}{382.5}{18}{18}
\cput{2.310,0.9568}{562.5}{-18}{-18}
\cput{-0.9567,-2.310}{427.5}{19}{19}
\cput{0.9567,2.310}{607.5}{-19}{-19}
\cput{0.9567,-2.310}{472.5}{20}{20}
\cput{-0.9568,2.310}{652.5}{-20}{-20}
\end{pspicture}  
\begin{pspicture}(-6,-5.5)(6,5.5) 
\pscircle[fillstyle=solid](0,0){2.50000000000000}
\pscircle(0,0){5}
\pspolygon[fillstyle=solid,fillcolor=gray](-0.4901,4.976)(0.4901,4.976)(1.451,4.785)(2.357,4.410)(3.172,3.865)(3.865,3.172)(4.410,2.357)(0.9567,2.310)
\pspolygon[fillstyle=solid,fillcolor=gray](0.4901,-4.976)(-0.4899,-4.976)(-1.452,-4.785)(-2.357,-4.410)(-3.172,-3.865)(-3.865,-3.172)(-4.410,-2.357)(-0.9567,-2.310)
\pscircle[fillstyle=solid](0,0){2.50000000000000}
\pscircle(0,0){5}
\pspolygon[fillstyle=solid,fillcolor=gray](4.785,1.451)(4.976,0.4901)(4.976,-0.4901)(2.357,-4.410)(1.451,-4.785)(0.9567,-2.310)(2.310,-0.9567)(2.310,0.9568)
\pspolygon[fillstyle=solid,fillcolor=gray](-4.785,-1.451)(-4.976,-0.4901)(-4.976,0.4900)(-2.357,4.410)(-1.451,4.785)(-0.9568,2.310)(-2.310,0.9567)(-2.310,-0.9567)
\pscircle[fillstyle=solid](0,0){2.50000000000000}
\pscircle(0,0){5}
\pspolygon[fillstyle=solid,fillcolor=gray](4.785,-1.451)(4.410,-2.357)(3.865,-3.172)(3.172,-3.865)
\pspolygon[fillstyle=solid,fillcolor=gray](-4.785,1.451)(-4.410,2.357)(-3.865,3.172)(-3.172,3.865)
\cput{0.4901,4.976}{84.38}{1}{1}
\cput{-0.4899,-4.976}{264.4}{-1}{-1}
\cput{1.451,4.785}{73.12}{2}{2}
\cput{-1.452,-4.785}{253.1}{-2}{-2}
\cput{2.357,4.410}{61.88}{3}{3}
\cput{-2.357,-4.410}{241.9}{-3}{-3}
\cput{3.172,3.865}{50.62}{4}{4}
\cput{-3.172,-3.865}{230.6}{-4}{-4}
\cput{3.865,3.172}{39.38}{5}{5}
\cput{-3.865,-3.172}{219.4}{-5}{-5}
\cput{4.410,2.357}{28.12}{6}{6}
\cput{-4.410,-2.357}{208.1}{-6}{-6}
\cput{4.785,1.451}{16.88}{7}{7}
\cput{-4.785,-1.451}{196.9}{-7}{-7}
\cput{4.976,0.4901}{5.625}{8}{8}
\cput{-4.976,-0.4901}{185.6}{-8}{-8}
\cput{4.976,-0.4901}{-5.625}{9}{9}
\cput{-4.976,0.4900}{174.4}{-9}{-9}
\cput{4.785,-1.451}{-16.88}{10}{10}
\cput{-4.785,1.451}{163.1}{-10}{-10}
\cput{4.410,-2.357}{-28.12}{11}{11}
\cput{-4.410,2.357}{151.9}{-11}{-11}
\cput{3.865,-3.172}{-39.38}{12}{12}
\cput{-3.865,3.172}{140.6}{-12}{-12}
\cput{3.172,-3.865}{-50.62}{13}{13}
\cput{-3.172,3.865}{129.4}{-13}{-13}
\cput{2.357,-4.410}{-61.88}{14}{14}
\cput{-2.357,4.410}{118.1}{-14}{-14}
\cput{1.451,-4.785}{-73.12}{15}{15}
\cput{-1.451,4.785}{106.9}{-15}{-15}
\cput{0.4901,-4.976}{-84.38}{16}{16}
\cput{-0.4901,4.976}{95.62}{-16}{-16}
\cput{-2.310,0.9567}{337.5}{17}{17}
\cput{2.310,-0.9567}{517.5}{-17}{-17}
\cput{-2.310,-0.9567}{382.5}{18}{18}
\cput{2.310,0.9568}{562.5}{-18}{-18}
\cput{-0.9567,-2.310}{427.5}{19}{19}
\cput{0.9567,2.310}{607.5}{-19}{-19}
\cput{0.9567,-2.310}{472.5}{20}{20}
\cput{-0.9568,2.310}{652.5}{-20}{-20}
\end{pspicture}  
 \caption{Two elements in $\NC_D^{(4)}(5)$.}
 \label{fig:annulus2}
\end{figure}

Let $\pi\in\nckd(n)$. If a nonzero block $B$ of $\pi$ contains integers both on
the outer circle and on the inner circle, we call $B$ an \emph{annular block} of
$\pi$. Let $\tnckd(n)$ denote the set of $\pi\in\nckd(n)$ containing at least
one annular block.

Consider an $(\ell+1)$-tuple $(L,R_1,\ldots,R_\ell)$ of sets
$L,R_1,\ldots,R_\ell\subset[n]$ with $|L|\leq|R_1|+\cdots+|R_\ell|$. Let
$R=R_1\uplus R_2\uplus\cdots\uplus R_\ell$ be the disjoint union of
$R_1,R_2,\ldots,R_\ell$. We regard the elements in $L$ and $R$ as left and right
parentheses of $1,2,\ldots,n$ in cyclic order. If $|L|< |R|$, then there are
several right parentheses, which do not have corresponding left parentheses. We
call such a right parenthesis \emph{unmatched}.  Let $R'$ be the set of
unmatched right parentheses. Let $f$ be a function from $R'$ to positive
integers. Then we can represent the $(\ell+2)$-tuple $(L,R_1,\ldots,R_\ell,f)$
as a \emph{doubly labeled cyclic parenthesization} as we did in
Section~\ref{sec:k-divis-noncr} together with the additional label $f(x)$ for
each unmatched right parenthesis $x$.

\begin{example}\label{ex:2}
Let $L=\{2,3,5,6,12,13\}$, $R_1=\{4,8,12\}$, $R_2=\{4,9\}$, $R_3=\{1\}$ and $R_4=\{4,7\}$. Then $(L,R_1,R_2,R_3,R_4)$ is represented by the following parenthesization:
\begin{equation}
  \label{eq:21}
1)_3\:\: (2\:\: (3\:\: 4)_1)_2)_4\:\: (5\:\: (6\:\: 7)_4\:\: 8)_1\:\: 9)_2\:\: 10\:\: 11\:\: (12)_1)_3\:\: (13
\end{equation}
The unmatched right parentheses are $4\in R_4$, $9\in R_2$ and $12\in R_3$. Let $f$ be the function on these unmatched right parentheses with values $3,5$ and $2$ respectively. Then $(L,R_1,R_2,R_3,f)$ is represented by the following doubly labeled cyclic parenthesization:
  \begin{equation}
    \label{eq:20}
1)_3\:\: (2\:\: (3\:\: 4)_1)_2)_4^3\:\: (5\:\: (6\:\: 7)_4\:\: 8)_1\:\: 9)_2^5\:\: 10\:\: 11\:\: (12)_1)_3^2\:\: (13
  \end{equation}
\end{example}

Let $P_D(n,\ell)$ denote the set of $(\ell+2)$-tuples $(L,R_1,\ldots,R_\ell, f)$ where $L,R_1,\ldots,R_\ell\subset[n]$, $|L|\leq |R_1|+\cdots+|R_\ell|$ and $f$ is a function from the unmatched right parentheses to positive integers. 

Let $(L,R_1,\ldots,R_\ell, f)\in P_D(n,\ell)$. Let  $R=R_1\uplus R_2\uplus\cdots\uplus R_\ell$ and let $R'$ be the set of unmatched right parentheses.
For each right parenthesis $x\in R$, we define the \emph{size of $x$} as follows. If $x\not\in R'$, then the size of $x$ is defined to be the number of integers enclosed by $x$ and its corresponding left parenthesis, which are not enclosed by any other pair of parentheses. If $x\in R'$, then the size of $x$ is defined to be $f(x)$ plus the number of integers between $x$ and its previous unmatched right parenthesis which are not enclosed by any pair of parentheses.

\begin{example}
Let $(L,R_1,R_2,R_3,R_4,f)$ be the element in $P_D(13,4)$ represented by \eqref{eq:20}. If we draw only the right parentheses labeled with their sizes, then we get the following:
  \begin{equation}
    \label{eq:22}
1)_3^2\:\: 2\:\: 3\:\: 4)_1^2)_2^1)_4^3\:\: 5\:\: 6\:\: 7)_4^2\:\: 8)_1^2\:\: 9)_2^6\:\: 10\:\: 11\:\: 12)_1^1)_3^4\:\: 13
 \end{equation}

Note that we can recover \eqref{eq:20} from \eqref{eq:22}.
\end{example}

Let $\pdk(n,\ell)$ denote the set of $(\ell+2)$-tuples $(L,R_1,\ldots,R_\ell, f)\in P_D(kn-k,\ell)$ with the following conditions. As before, $R=R_1\uplus R_2\uplus\cdots\uplus R_\ell$ and
$R'$ is the set of unmatched right parentheses.
\begin{enumerate}
\item If $|L|=|R|$, then there is at least one integer which is not enclosed by any pair of parentheses.
\item If $|L|<|R|$, then $\sum_{x\in R'}f(x)=k$.
\item The size of each right parenthesis $x\in R$ is divisible by $k$.
\end{enumerate}

For $(L,R_1,\ldots,R_\ell, f) \in\pdk(n,\ell)$, let $\typek(L,R_1,\ldots,R_\ell, f) = (b;b_1,b_2,\ldots,b_n)$, where $b_i$ is the number of right parentheses whose sizes are equal to $ki$ for $i\in[n]$, and $b=b_1+b_2+\cdots+b_n$. Note that if $|L|=|R|$ then $\sum_{i=1}^n i\cdot b_i\leq n-2$, and if $|L|<|R|$ then $\sum_{i=1}^n i\cdot b_i=n$.
 
\begin{lem}
  \label{thm:17}
Let $b,b_1,b_2,\ldots,b_n$ and $c_1,c_2,\ldots,c_\ell$ be nonnegative integers
with $b=b_1+b_2+\cdots+b_n=c_1+c_2+\cdots+c_\ell$, $\sum_{i=1}^n i\cdot
b_i\leq n$ and $\sum_{i=1}^n i\cdot b_i\neq n-1$. Then the number of elements $P=(L,R_1,\ldots,R_\ell, f)$ in $\pdk(n,\ell)$ with $\typek(P)=(b;b_1,b_2,\ldots,b_n)$ and $(|R_1|,|R_2|,\ldots,|R_\ell|)=(c_1,c_2,\ldots,c_\ell)$ is equal to
$$\binom{b}{b_1,b_2,\ldots,b_n} \binom{k(n-1)}{c_1} \cdots \binom{k(n-1) }{c_{\ell}}.$$
\end{lem}
\begin{proof}
This can be done by the same argument as in the proof of Lemma~\ref{thm:16}.
\end{proof}

Given $P=(L,R_1,\ldots,R_\ell, f) \in\pdk(n,\ell)$, we define $\taud(P)$ as
follows.  Let $R=R_1\uplus R_2\uplus\cdots\uplus R_\ell$ and $R'$ be the set of
unmatched right parentheses. We will first define $\pi\in\nckd(n)$ as follows.
We place a left parenthesis before each occurrence of $i$ and $-i$ for each
$i\in L$ and right parentheses with subscripts $j_1,j_2,\ldots,j_t$ after each
occurrence of $i$ and $-i$ if $i$ is in $R_{j_1}, R_{j_2},\ldots,R_{j_t}$ with
$j_1<j_2<\cdots<j_t$ in the following infinite cyclic sequence:
\begin{equation}
  \label{eq:18}
\ldots,1,2,\ldots,k(n-1), -1,-2,\ldots,-k(n-1), 1,2,\ldots  
\end{equation}

Find a matching pair of parentheses which do not enclose any other parenthesis in \eqref{eq:18}. Remove the pair of matching parentheses and the integers enclosed by them and all of their infinite copies. The removed integers form a block of $\pi$. We repeat this procedure until there is no matching pair of parentheses. Then we have the following three possibilities. Now consider the $(2k(n-1),2k)$-annulus.

\begin{description}
\item[Case 1] There is no unmatched right parenthesis. This happens when $|L|=|R|$. By definition of $\pdk(n,\ell)$, there are remaining integers. Then these integers together with the integers on the inner circle form a zero block of $\pi$. We define $\taud(P)$ to be the set $\{\pi\}$ consisting of $\pi$ alone.
\item[Case 2] There is an unmatched right parenthesis, and there is no remaining
  integer.  Since the size of each right parenthesis is divisible by $k$, there
  must be only one unmatched right parenthesis $x$ with $f(x)=k$. Assume that
  $x$ was a right parenthesis of $a$. Let $b$ be the unique positive integer
  with $a\equiv b \mod k$ in the inner circle of the $(2k(n-1),2k)$-annulus.
  Then we add the two blocks $B$ and $-B$ to $\pi$, where the elements of $B$
  are the $k$ consecutive integers on the inner circle ending with $b$ in
  counterclockwise order.  We define $\taud(P)$ to be the set $\{\pi\}$
  consisting of $\pi$ alone.
\item[Case 3] There are unmatched right parentheses and remaining integers. Let $Y$ be the set of remaining integers. Let $R'=\{x_1,x_2,\ldots,x_j\}$ with $x_1<x_2<\cdots<x_j$. Then the right parentheses coming from $R'$ divide the initial infinite cyclic sequence \eqref{eq:18} into the $2j$ blocks $\pm X_i$'s, where
  \begin{align*}
    X_1 &= \{-x_j+1,\ldots,-k(n-1),1,2,\ldots,x_1\},\\
   X_i &= \{x_{i-1}+1, x_{i-1}+2,\ldots,x_i\} \quad \mbox{for $i=2,3,\ldots,j$.}
  \end{align*}
  For $i\in[j]$, we claim that $X_i\cap Y\ne\emptyset$. If $X_i\cap
  Y=\emptyset$, then the unmatched right parenthesis $x_i$ has size $f(x_i)$
  which is divisible by $k$. Since $\sum_{x\in R'} f(x)=k$ and $f(x)>0$, we get
  $f(x_i)=k$, $j=1$, and $Y=\emptyset$, which is a contradiction to the
  assumption in this case. Thus we have $X_i\cap Y\ne\emptyset$.

  Let $A_i=X_i\cap Y$ and $a_i=\max(A_i)$ for $i\in[j]$. We define $B_i^+$ and
  $B_i^-$ for $i\in[j]$ as follows. We first define $B_j^+$ and $B_j^-$. There
  is a unique positive integer $c_j$ on the inner circle such that $c_j\equiv
  a_j+1 \mod k$. Let $B_j^+$ (resp.~$B_j^-$) be the set obtained from $A_j$ by
  adding $f(a_j)$ consecutive integers on the inner circle starting from $c_j$
  (resp.~$-c_j$) in counterclockwise order. We define $B_{j-1}^+$
  (resp.~$B_{j-1}^-$) to be the set obtained from $A_{j-1}$ by adding the
  $f(a_{j-1})$ consecutive integers next to the integers used for $B_{j}^+$
  (resp.~$B_j^-$) on the inner circle. Repeat this until we get $B_1^+$ and
  $B_1^-$. Let $\pi^+$ (resp.~$\pi^-$) be the partition obtained from $\pi$ by
  adding $\pm B_1^+, \pm B_2^+,\ldots,\pm B_j^+$ (resp.~$\pm B_1^-, \pm
  B_2^-,\ldots,\pm B_j^-$). We define $\taud(P)$ to be the set
  $\{\pi^+,\pi^-\}$.
\end{description}

\begin{example}
  Let $X$ be the element in $P_D^{(4)}(16)$ represented by 
\begin{equation}
  \label{eq:25}
 (1 \:\: 2 \:\: 3 \:\: 4)_2 \:\: (5 \:\: 6 \:\: (7 \:\: 8 \:\: 9 \:\: 10)_2 \:\: 11 \:\: 12)_1)_2^4 \:\: (13 \:\: 14 \:\: 15 \:\: 16)_2.
\end{equation}
Then $\tau_D^{(4)}(X)=\{\pi\}$, where $\pi$ is the partition represented by the
left diagram in Figure~\ref{fig:annulus1}. Let $Y$ be the element in
$P_D^{(4)}(16)$ represented by 
\begin{equation}
  \label{eq:24}
  1 \:\: 2 \:\: (3 \:\: 4 \:\: 5 \:\: 6)_1)_2^1 \:\: 7 \:\: 8)_1^2 \:\: 9 \:\: (10 \:\: 11 \:\: 12 \:\: 13)_2 \:\: 14 \:\: 15)_2^1 \:\: (16.
\end{equation}
Then $\tau_D^{(4)}(Y)=\{\pi^+,\pi^-\}$, where $\pi^+$ is the partition represented by the left diagram in Figure~\ref{fig:annulus2} and $\pi^-$ is the partition obtained from $\pi^+$ by changing $i$ to $-i$ for each $i$ on the inner circle. 
\end{example}

\begin{prop}\label{thm:20}
Let $\pi\in\nckd(n)$. Then there is a unique $(L,R,f)\in\pdk(n,1)$ such that $\pi\in\taud(L,R,f)$.
\end{prop}
\begin{proof}
  The uniqueness is obvious from the construction of $\taud$. We find $(L,R,f)$
  as follows. Find a nonzero block $B$ of $\pi$ such that the integers in $B$ on
  the outer circle of the annulus are consecutive in clockwise order. If $B$
  does not have an integer on the inner circle, then add $|i|$ to $L$ and $|i'|$
  to $R$ where $i$ and $i'$ are the first and the last integers among the
  consecutive integers. If $B$ has an integer on the inner circle, then we only
  add $|i'|$ to $R$, and define $f(|i'|)$ to be the number of integers of
  $B$ on the inner circle. Remove the integers in $B$ and $-B$ in $\pi$, and
  repeat this until $\pi$ has no nonzero blocks. After finishing this procedure,
  we get $L,R$ and $f$. It is easy to see that $\pi\in\taud(L,R,f)$.

\end{proof}

\begin{example}
  The unique $(L,R,f)$'s for the partitions represented by the left and right
  diagrams in Figures~\ref{fig:annulus1} and \ref{fig:annulus2} are the
  following:
  \begin{equation}
    \label{eq:23}
 (1 \:\: 2 \:\: 3 \:\: 4) \:\: (5 \:\: 6 \:\: (7 \:\: 8 \:\: 9 \:\: 10) \:\: 11 \:\: 12) \:\: (13 \:\: 14 \:\: 15 \:\: 16)
 \end{equation}
\begin{equation}
   \label{eq:26}
 1 \:\: 2 \:\: 3 \:\: 4) \:\: 5 \:\: 6 \:\: (7 \:\: 8 \:\: 9 \:\: 10) \:\: 11 \:\: 12)^4 \:\: (13 \:\: 14 \:\: 15 \:\: 16
\end{equation}
\begin{equation}
   \label{eq:27}
  1 \:\: 2)^1 \:\: (3 \:\: 4 \:\: 5 \:\: 6) \:\: 7 \:\: 8)^2 \:\: 9 \:\: (10 \:\: 11 \:\: 12 \:\: 13) \:\: 14 \:\: 15)^1 \:\: (16
\end{equation}
\begin{equation}
   \label{eq:28}
  1 \:\: 2 \:\: 3 \:\: 4 \:\: 5 \:\: 6)^1 \:\: 7 \:\: 8 \:\: 9 \:\: (10 \:\: 11 \:\: 12 \:\: 13) \:\: 14 \:\: 15)^3 \:\: (16
\end{equation}
\end{example}

Let $\ovpk(n,\ell)$ denote the set of elements $(L,R_1,\ldots,R_\ell,f)$ in $\pdk(n,\ell)$ such that $|L|<|R_1|+\cdots+|R_\ell|$. 
For $P\in\ovpk(n,\ell)$ and $\epsilon\in\{-1,1\}$, we define a multichain $\pi_1\leq\pi_2\leq\cdots\leq\pi_\ell$ in $\nckd(n)$ as follows.

Let $P=(L,R_1,\ldots,R_\ell, f)$. Consider the doubly labeled parenthesization
representing $P$ as shown in \eqref{eq:22}. Then we define $P_1,P_2,\ldots,P_n$
as follows. Let $P_1=P$. For $2\leq i\leq
n$, $P_i$ is obtained from $P_{i-1}$ by removing all the right parentheses in
$R_{i-1}$ and their corresponding left parentheses, if they exist. If unmatched right
parentheses are removed, then their $f$ values are added to the next unmatched
right parentheses in cyclic order, if they exist. For instance, if $P_1$ is
$$1)_2^3\:\: 2\:\: 3)_1^2\:\: 4\:\: 5)_1^1\:\: 6\:\: (7\:\: 8)_1)_2^1\:\: 9\:\: 10)_1^2\:\: 11\:\: (12)_2\:\:,$$
then, in $P_2$ the $f$ values of $3,5\in R_1$ are added to the $f$ value of $8\in R_2$, and the $f$ value of $10\in R_1$ is added to the $f$ value of $1\in R_2$. Thus $P_2$ is
$$1)_2^5\:\: 2\:\: 3\:\: 4\:\: 5\:\: 6\:\: 7\:\: 8)_2^4\:\: 9\:\: 10\:\: 11\:\: (12)_2\:\:.$$
Now consider the sequence of sets
$\tau_D(P_1),\tau_D(P_2),\ldots,\tau_D(P_\ell)$. Observe that at least one of
these sets contains two elements because $P\in\ovpk(n,\ell)$. 

Let $i$ be the smallest integer such that $\tau_D(P_i)$ has two elements. Then
$\tau_D(P_i)=\{\pi^+,\pi^-\}$. Let $\pi_1,\pi_2,\ldots,\pi_{i-1}$ be the unique
elements in $\tau_D(P_1)$, $\tau_D(P_2),\ldots$, $\tau_D(P_{i-1})$
respectively. Let $\pi_i=\pi^+$ if $\epsilon=1$; and $\pi_i=\pi^-$ if
$\epsilon=-1$. There is a unique sequence $\pi_{i+1},\pi_{i+2},\ldots,\pi_\ell$
of elements in $\tau_D(P_{i+1})$, $\tau_D(P_{i+2}),\ldots$, $\tau_D(P_{\ell})$
respectively such that $\pi_i\leq\pi_{i+1}\leq\cdots\leq\pi_\ell$. We define
$\tau_D'(P,\epsilon)$ to be the multichain
$\pi_1\leq\pi_2\leq\cdots\leq\pi_\ell$ in $\nckd(n)$. Note that we have
$\pi_j\in\tnckd(n)$ for at least one $j\in[\ell]$.

\begin{remark}
Our map $\tau_D'$ is a generalization of the map of Athanasiadis and Reiner for $\ncd(n)$ defined in \cite[Proposition 4.3]{Athanasiadis2005}. 
\end{remark}

\begin{example}
  If $P$ is \eqref{eq:25} (resp.~\eqref{eq:24}), then $\tau_D'(P)$ is the multichain consisting of two elements in Figure~\ref{fig:annulus1} (resp.~Figure~\ref{fig:annulus2}). 
\end{example}

\begin{prop}
  \label{thm:19}
The map $\tau_D'$ is a bijection between $\ovpk(n,\ell)\times\{-1,1\}$ and the set of multichains $\pi_1\leq\pi_2\leq\cdots\leq\pi_\ell$ in $\nckd(n)$ with $\pi_j\in\tnckd(n)$ for at least one $j\in[\ell]$. Moreover, if $P=(L,R_1,\ldots,R_\ell,f)$ and $\tau_D'(P,\epsilon)$ is the multichain
$\pi_1\leq\pi_2\leq\cdots\leq\pi_\ell$ with rank jump vector $(s_1,s_2,\ldots,s_{\ell+1})$, then $\typek(P)=\typek(\pi_1)$ and
$(|R_1|,|R_2|,\ldots,|R_\ell|)=(s_2,s_3,\ldots,s_{\ell+1})$. 
\end{prop}
\begin{proof}
  We will find the inverse map of $\tau_D'$. Let
  $\pi_1\leq\pi_2\leq\cdots\leq\pi_\ell$ be a multichain in $\nckd(n)$ with
  $\pi_j\in\tnckd(n)$ for at least one $j\in[\ell]$. By
  Proposition~\ref{thm:20}, there is a unique element
  $(L_i',R_i',f_i)\in\pdk(n,1)$ with $\pi_i\in \tau_D(L_i',R_i',f_i)$ for each
  $i\in[\ell]$. Let $j$ be the smallest integer such that
  $\pi_j\in\tnckd(n)$. Then $\pi_j\in\tau_D(L_j',R_j',f_j)=\{\pi^+,\pi^-\}$. We
  define $\epsilon=1$ if $\pi_j=\pi^+$ and $\epsilon=-1$ otherwise.

  Now we define $P=(L,R_1,R_2,\ldots,R_\ell,f)\in\ovpk(n,\ell)$ as
  follows. First we set $L=L_\ell'$ and $R_\ell=R_\ell'$. Then define
  $R_{\ell-1}$ to be the set of right parentheses in $R_{\ell-1}'$ whose
  corresponding left parentheses, if they exist, are not already in $L$. We add
  these corresponding left parentheses to $L$. Repeat this procedure until we
  get $L,R_1,R_2,\ldots,R_\ell$. Note that the unmatched right parentheses of
  $R=R_1\uplus R_2\uplus \cdots \uplus R_\ell$ are exactly those of $R_1'$. Let
  $f=f_1$.

It is straightforward to check that the map sending $\pi_1\leq\pi_2\leq\cdots\leq\pi_\ell$ to $(P,\epsilon)$ is the inverse map of $\tau_D'$. The `moreover' statement is obvious from the construction of $\tau_D'$. 
\end{proof}

\begin{lem}\label{thm:11}
Let $b,b_1,b_2,\ldots,b_n$ and $s_1,s_2,\ldots,s_{\ell+1}$ be nonnegative integers satisfying $\sum_{i=1}^n b_i=b$, $\sum_{i=1}^n i\cdot b_i=n$, $\sum_{i=1}^{\ell+1} s_i = n$ and $s_1=n-b$. Then the number of multichains $\pi_1\leq \pi_2\leq \cdots \leq \pi_{\ell}$ in $\nckd(n)$ with rank jump vector $(s_1,s_2,\ldots,s_{\ell+1})$ and $\typek(\pi_1)=(b;b_1,b_2,\ldots,b_n)$ such that $\pi_j\in\tnckd(n)$ for at least one $j\in[\ell]$ is equal to
$$2\binom{b}{b_1,b_2,\ldots,b_n} \binom{k(n-1)}{s_2} \cdots \binom{k(n-1)}{s_{\ell+1}}. $$
\end{lem}
\begin{proof}
  By Proposition~\ref{thm:19}, the number of such multichains is equal to $2$ times the number of $P=(L,R_1,\ldots,R_\ell,f)\in\ovpk(n,\ell)$ with $\typek(P)=(b;b_1,b_2,\ldots,b_n)$ and $(|R_1|,|R_2|,\ldots,|R_\ell|)=(s_2,s_3,\ldots,s_{\ell+1})$. Thus we are done by Lemma~\ref{thm:17}. Note that by the condition $\sum_{i=1}^n i\cdot b_i=n$, if $\typek(P)=(b;b_1,b_2,\ldots,b_n)$, then $P\in\ovpk(n,\ell)$ is equivalent to $P\in\pdk(n,\ell)$.
\end{proof}

Now we can prove Theorem~\ref{thm:kratt3}.

\begin{proof}[Proof of Theorem~\ref{thm:kratt3}]
  Let $\pi_1\leq \pi_2\leq \cdots \leq \pi_{\ell}$ be a multichain satisfying
  the conditions.

  Assume that $b_1+2b_2+\cdots+n b_n \leq n-2$.   Then $\pi_1$ has
  a zero block and $\pi_j\in\tnckd(n)$ for each $j\in[\ell]$. Let $\pi_j'$ be
  the element in $\nckb(n-1)$ obtained from $\pi_j$ by removing all the integers
  on the inner circle. Then $\pi_1'\leq \pi_2'\leq \cdots\leq \pi_{\ell}'$ is a
  multichain in $\nckb(n-1)$ with rank jump vector
  $(s_1-1,s_2,\ldots,s_{\ell+1})$ and $\typek(\pi_1')=(b;b_1,\ldots,b_n)$. Thus
  by Theorem~\ref{thm:kratt2}, the number of such
  multichains is equal to
$$\binom{b}{b_1,b_2,\ldots,b_n} \binom{k(n-1)}{s_2} \cdots \binom{k(n-1)}{s_{\ell+1}}.$$

Now assume that $b_1+2b_2+\cdots+n b_n = n$.  By Lemma~\ref{thm:11}, the number
of multichains with $\pi_j\in\tnckd(n)$ for at least one $j\in[\ell]$ is equal
to
$$2\binom{b}{b_1,b_2,\ldots,b_n} \binom{k(n-1)}{s_2} \cdots \binom{k(n-1)}{s_{\ell+1}}. $$
The remaining case is that $\pi_j\not\in\tnckd(n)$ for all $j\in[\ell]$. Let
$\pi_j'$ be the element in $\nckb(n-1)$ obtained from $\pi_j$ by removing all
the integers on the inner circle. Then $\pi_1'\leq \pi_2'\leq \cdots\leq
\pi_{\ell}'$ is a multichain in $\nckb(n-1)$.  Let $i$ be the index of this
chain, i.e.~the smallest integer such that $\pi_i'$ has a zero block. Since
$b_1+2b_2+\cdots+n b_n = n$, we must have $i\geq2$. The rank jump vector of
$\pi_1'\leq \pi_2'\leq \cdots\leq \pi_{\ell}'$ is
$(s_1,\ldots,s_i-1,\ldots,s_{\ell+1})$ and
$\typek(\pi_1)=(b-1;b_1-1,\ldots,b_n)$. By Lemma~\ref{thm:atha}, the number of
such multichains is equal to
$$\frac{s_i-1}{b-1} \binom{b-1}{b_1-1,b_2,\ldots,b_n} \binom{k(n-1)}{s_2} \cdots \binom{k(n-1)}{s_i-1}\cdots\binom{k(n-1)}{s_{\ell+1}}.$$
Summing the above formula for $i=2,3,\ldots,\ell+1$, we finish the proof.
\end{proof}

\section*{Acknowledgement}

The author is grateful to Christian Krattenthaler for providing the correct
definition of $\nckd(n)$. He is also grateful to the anonymous referees for
their careful reading and helpful comments.

\bibliographystyle{abbrv}

\end{document}